\documentclass[11pt,a4paper,twoside]{article}

\usepackage{amssymb,bbm,amsfonts,amsmath,amsthm,euscript,a4wide,xargs}
\usepackage[english]{babel}
\usepackage[utf8]{inputenc}
\usepackage{dsfont}
\usepackage{enumitem}
\usepackage[bf,small]{caption}

\usepackage{hyperref}
 \hypersetup{colorlinks=true,citecolor=blue, urlcolor=blue, linkcolor=black, filecolor=black}

\def\rmi{\mathrm i}

\def\rme{\mathrm e}

\def\rmd{\mathrm d}

\def\var{\mathrm{Var}}
\def\1{\mathbbm 1}
\newcommand{\cum}[1]{\mathrm{Cum}\left(#1\right)}
\newcommand{\past}[1]{#1_{(-)}}
\newcommand{\fut}[1]{#1_{(+)}}
\newcommand{\futpast}[1]{#1_{(\pm)}}

\def\closure{\mathrm{Cl}}

\def\bandwidth{\beta}

\def\tore{\mathbb{T}}
\newcommand{\intpermu}{\mathrm{L}}
\newcommand{\intper}[1]{\int\,#1\,\rmd \intpermu}
\newcommand{\intperarg}[2]{\int\,#1\,\intpermu(\rmd #2)}
\def\1{\mathbbm{1}}
\def\cont{\mathcal{C}}
\def\R{\mathbb{R}}
\def\rset{\mathbb{R}}

\def\N{\mathbb{N}}
\def\nset{\mathbb{N}}

\def\E{\mathbb{E}}
\def\L{\mathbb{L}}

\def\C{\mathbb{C}}
\def\cset{\mathbb{C}}
\def\Z{\mathbb{Z}}

\def\v{\mbox{\rm Var}}
\newcommand{\abs}[1]{\left|#1\right|}

\def\cov{\mbox{\rm Cov}}

\newtheorem{theo}{Theorem}
\newtheorem{lem}{Lemma}
\newtheorem{prop}{Proposition}
\newtheorem{Def}{Definition}
\newtheorem{cor}{Corollary}

\newtheorem{Rem}{Remark}

\newcommand {\noi}{\noindent}

\begin{document}

\author{P. Doukhan\footnote{University Cergy-Pontoise, AGM UMR8088 and CIMFAV, University of Valparaiso.}, \
F. Roueff\footnote{LTCI, Telecom Paris, Institut Polytechnique de Paris},
 \ and \
J. Rynkiewicz\footnote{University Paris 1, la Sorbonne, SAMM EA4543.}  
}

\title{Spectral estimation for non-linear long range dependent
  discrete time trawl processes}

 \maketitle

\begin{abstract}\noi
  Discrete time trawl processes constitute a large class of time
  series parameterized by a trawl sequence $(a_j)_{j\in\N}$ and
  defined though a sequence of independent and identically distributed
  (i.i.d.) copies of a continuous time process $(\gamma(t))_{t\in\R}$
  called the seed process. They provide a general framework for
  modeling linear or non-linear long range dependent time series. We
  investigate the spectral estimation, either pointwise or broadband,
  of long range dependent discrete-time trawl processes. The
  difficulty arising from the variety of seed processes and of trawl
  sequences is twofold. First, the spectral density may take different
  forms, often including smooth additive correction terms. Second,
  trawl processes with similar spectral densities may exhibit very
  different statistical behaviors. We prove the consistency of our
  estimators under very general conditions and we show that a wide
  class of trawl processes satisfy them. This is done in particular by
  introducing a weighted weak dependence index that can be of
  independent interest. The broadband spectral estimator includes an
  estimator of the long memory parameter. We complete this work with
  numerical experiments to evaluate the finite sample size performance
  of this estimator for various integer valued discrete time trawl
  processes.
\end{abstract}

\medskip

\noi {\it Keywords:} trawl processes; integer-valued time series; long
memory parameter estimation

\noi {\it MSC:}  62M10; 62F12; 60G51;  

\section{Introduction}\label{sec1}
A \emph{discrete time
  trawl} process  $X=\{ X_k,\,k\in\Z\}$ is defined in  \cite{DJLS-spa} by 
\begin{equation}\label{g1}
X_k\ =\  \sum_{j=0}^\infty \gamma_{k-j}(a_j), \qquad k \in \Z,
\end{equation}
where
\begin{enumerate}[label=(A-\arabic*),series=hyp,start=1]
\item\label{item:basic} The sequence $(\gamma_k)_{k\in\Z}$ is a sequence of
  i.i.d. copies of a generic process $\gamma = \{\gamma (u), u \in \R\}$ and
  $a=\{a_j,\, j\geq0\}$ is a sequence converging to zero.
\end{enumerate}
Processes so defined  can be interpreted as discrete time versions of the
trawl processes introduced in \cite{bar2014}. The generic process $\gamma$ is
called the \emph{seed process} and the sequence $a$ is called the \emph{trawl
(height)  sequence}.

Additional assumptions are required to have a converging sum in~(\ref{g1}).
The convergence in $\L^{2}$ is guaranteed if
\begin{equation}
  \label{eq:hyp-a1}
  \sum_{j=0}^\infty\left(
  \left|\E\,\gamma(a_j)\right|+\v\,\gamma(a_j)\right)<\infty \;.
\end{equation}
See \cite[Proposition~1]{DJLS-spa}, where the covariance function
  is also given by the formula
 \begin{align}
    \label{eq:covariance-trawl}
  r(k)=\cov (X_0,X_k) =\sum_{j=0}^\infty\cov(\gamma(a_j),\gamma(a_{j+k}))\;.
 \end{align}
 By \cite[Proposition~3]{DJLS-spa}, we moreover know that if, in addition, the two
  following asymptotic behaviors hold:
  \begin{align}
    \label{g3} 
  \cov(\gamma(u),\gamma(v))&=\left(|u|\wedge|v|\right)\;(1+o(1))\quad\text{as}\quad
                             u,v\to0\;,\\
    \label{g4} 
  a_j \ &= \  c \ j^{-\alpha^*} (1+ o(1))\quad\text{as}\quad j \to \infty  \;,
\end{align}
with $c\neq0$ and $\alpha^*>1$, then, the covariance function behaves at large
lags as
\begin{equation}
  \label{eq:r-k-asymp}
r(k)= c'\ k^{1-\alpha^*}\ (1+o(1))\quad\text{as}\quad k \to \infty  \;.
\end{equation}
 In the following we will refer to $\alpha^*$ in~(\ref{g4}) as the
\emph{trawl exponent}. In particular if
\begin{equation}
  \label{eq:alpha-lrd}
1<\alpha^*<2 \;,  
\end{equation}
this behavior is
often referred to as $X$ being \emph{long range dependent} with long memory
parameter
\begin{equation}
  \label{eq:d-alpha}
d^*=1-\alpha^*/2 \;.  
\end{equation}
Here~(\ref{eq:alpha-lrd}) implies $d^*\in(0,1/2)$ (sometimes referred to as
\emph{positive} long memory).  We here use one of the several existing
definitions of long range dependence, see for instance Condition~II in
\cite[Section~2.1]{pipiras_taqqu_2017}. In fact in the cases considered here,
the same long memory parameter $d^*$ can also be defined through their
condition~IV, based on the spectral density.  In the case where
$\alpha^*\geq2$, the two definitions may no longer coincide. The definition of
\emph{negative} long memory ($d^*<0$) is generally relying on the behavior of
the spectral density at the origin (in particular imposing this spectral
density to vanish there). Adopting this definition the
formula~(\ref{eq:d-alpha}) may not be valid anymore as the obtained process
could have short memory ($d^*=0$) even if $\alpha^*>2$, or have negative long
memory ($d^*<0$). In the following, we will only consider the case where
$d^*\geq0$, avoiding the negative long memory case for convenience.

A very interesting feature of trawl processes is that under the fairly
general assumption~(\ref{g3}) on the seed process, the low frequency behavior
of the spectral density is mainly driven by the trawl sequence.
However, it is shown in
\cite{DJLS-spa} that, for a given trawl sequence, two different seed
processes can yield different large scale behaviors, as can be seen by
different types of limits in the invariance
principle. In the case of a L{\'e}vy seed for instance, a Brownian seed process
leads to an invariance principle with fractional Brownian motion limit, with
Hurst parameter $(3-\alpha^*)/2$, and a (centered) Poisson seed process leads
to an invariance principle with L{\'e}vy $\alpha^*$-stable limit, see
\cite[Theorems~1 and~2]{DJLS-spa}. 

The goal of this paper is to investigate the spectral estimation of a
long-range dependent process $X$ from a sample
$X_1,\dots,X_n$. Deriving general results applying to a wide class of
long range dependent trawl processes raise two major
difficulties. First, as already noted about the asymptotic results
derived in \cite{DJLS-spa}, the large scale behavior of such
processes, can be very different from one trawl process to another,
even with similar or even equal covariance structure. Second, the
spectral density has a closed form only in particular cases for the
seed process and the trawl sequence. The computation of the spectral
density function depends both on the seed process $\gamma$ and the
sequence $(a_j)$.  For instance, in \cite[Example~5]{DJLS-spa}, it is
shown that for a large class of seed processes (that will be referred
to as the Lévy seed process below), a specific sequence $(a_j)$ leads
to the same spectral density as an ARFIMA(0,$d^*$,0), namely,
\begin{equation}
    \label{eq:arfima}
f_{d^*}(\lambda)=\frac1{2\pi}\,\left|1-\rme^{-\rmi\lambda}\right|^{-2d^*}\;.
\end{equation}
Here the spectral density is normalized in such a way to have the
innovation process with unit variance. The general form that we will
assume on the spectral density includes of course a
multiplicative constant $c^*$ but also an additive smooth function
$h^*$ belonging to the space $\cont$ of continuous and $(2\pi)$
periodic functions endowed with the sup norm. Namely, to encompass as
many cases as possible, we assume that $X$ has a spectral density
function given by
\begin{equation}
  \label{eq:sp-dens-general}
f(\lambda)=c^* \, \left(f_{d^*}(\lambda)+h^*(\lambda)\right)\;,\qquad\lambda\in\R\;, 
\end{equation}
where $d^*\in[0,1/2)$, $c^*>0$ and $h^*\in\cont$. The
form~(\ref{eq:sp-dens-general}) is the spectral behavior corresponding to that
of the covariance in~(\ref{eq:r-k-asymp}). Here $d^*$ is again the long memory
parameter, and it characterizes the power law behavior of $f$ at low
frequencies while the function $h^*$ encompasses the short-range behavior. As
we will see, in many cases of interesting trawl processes, the function $h^*$
is smooth in the Hölder sense, leading naturally to the
additive parametric form~(\ref{eq:sp-dens-general}) of the spectral density, which is different for
the usual product parametric form usually encountered in linear models such as ARFIMA
processes. Note however that such an additive form of the spectral density
were already considered in \cite{hurvich05_lmsv} for completely different
(non-linear) models.

We consider either pointwise or broadband estimation of the spectral
density. In the first case, we estimate $f(\lambda)$ directly for a
given $\lambda$, and, in the second case, we estimate the triplet
$(c^*,d^*,h^*)$ by assuming it belongs to a known parameter set. The
first approach only makes sense for $\lambda\neq0$ and will be
investigated in Section~\ref{sec:second-order-estim} using a smoothed
version of the periodogram. The second approach will be investigated
in Section~\ref{sec:param-whittle-estim}. The estimation of the long
memory parameter is a widely studied problem in statistical inference,
see the reference book \cite{DOT03}, or, more recently, \cite{gir2012}
and the references therein.  Here, we propose to estimate the
parameter $(c^*,d^*,h^*)$ using a parametric Whittle approach. Define
the periodogram
\begin{equation}\label{eq:periodo-def}
  I_n(\lambda)=\frac1{2\pi n}\left|\sum_{k=1}^n(X_k-\bar X_n)\rme^{-\rmi\lambda k}\right|^2\;,
\end{equation}
where $\bar X_n$ denotes the empirical mean of the sample $X_1,\dots,X_n$, and
denote the Whittle contrast by
\begin{equation}
  \label{eq:whittle-contrast}
  \Lambda_n(d,h)=
  \ln\left(\intper{\frac{I_n}
  {f_d+h}}\right)\;+\intper{\ln \left(f_d+h\right)}
\; ,
\end{equation}
where $f_{d}$ is defined by~(\ref{eq:arfima}) and $\intpermu$ is the Lebesgue
measure on $[-\pi,\pi]$ divided by $2\pi$.
Our estimator $(\hat d_{n},\hat h_{n}, \hat c_{n})$ is to find a near minimizer
$(\hat d_{n},\hat h_{n})$ of
$(d,h)\mapsto\Lambda_n(d,h)$ over a well chosen set of parameters for $(d,h)$,
and then set
\begin{align}
    \label{eq:c-parm-estim}
                            \hat c_{n}&=\intper{\frac{I_n}
                                          { f_{\hat d_{n}}+\hat h_{n}}}\;.
\end{align}
From which we can also define an estimator of the spectral density, namely,
$$
\hat f_{n}= \hat c_{n}\ \left( f_{\hat d_{n}}+\hat h_{n}\right)\;.
$$
Here we  derive results that apply to a wide class of trawl
processes, in particular to those of nature quite different  from
the well studied class of Gaussian or linear processes. For
convenience, we focus on proving the consistency of our estimators
under very general assumptions, that can be of
interest beyond trawl processes:
\begin{enumerate}[resume*=hyp]
\item \label{item:hyp-consistency-ergo} The process $X=(X_k)_{k\in\Z}$ is stationary, ergodic
  and $L^2$.
\item\label{item:hyp-consistency-common} There exist $C_0>0$ and $s_0\in(0,1)$ such that,
  for all integers $t_1\leq t_2 \leq t_3\leq t_4$,
  \begin{align}\label{eq:cov-gen-ass-bound}
\abs{\cov(X_{t_1}\,,\,X_{t_2})}&\leq
                             C_0\,(1+t_2-t_1)^{-s_0}      \;,\\
    \label{eq:pairwise-cov-gen-ass-bound}
    \abs{\cov(X_{t_1}X_{t_2}\,,\,X_{t_3}X_{t_4})}&\leq
                                               C_1\,(1+t_3-t_2)^{-s_0}      \;,\\
    \label{eq:13wise-cov-gen-ass-bound}
    \abs{\cov(X_{t_1}X_{t_2}X_{t_3}\,,\,X_{t_4})}&\leq
                                               C_1\,(1+t_4-t_3)^{-s_0}      \;.
  \end{align}
\end{enumerate}
Assumption~\ref{item:hyp-consistency-ergo} is basically satisfied by
all well defined discrete-time trawl processes. To show that a given
trawl process satisfies~\ref{item:hyp-consistency-common} with a well
chosen exponent $s_0$, we will rely on a weighted weak dependence
property that is easy to prove for discrete-time trawl processes.

The paper is organized as follows. In Section~\ref{sec:main-results}, we
present successively: 1) general conditions on the seed process and the trawl
sequence so that the corresponding trawl process satisfies
Condition~\ref{item:hyp-consistency-ergo} and~\ref{item:hyp-consistency-common}
above, 2) general results on second order estimation under
Assumption~\ref{item:hyp-consistency-common} and 3) a general consistency
result on the parametric Whittle estimation of the parameters $(d^*,h^*)$ of
the unknown spectral density in~(\ref{eq:arfima}). For this estimation result
to hold, we only require on the observed process to
satisfy~\ref{item:hyp-consistency-ergo}. The assumption on the parameter set on
which the Whittle contrast is maximized will be detailed
in~\ref{item:hyp-consistency}. We provide in Section~\ref{sec:example-trawl}
various examples of trawl processes. Although the usual causal linear models
for long range dependence (such as ARFIMA processes) constitute specific
examples of trawl processes, we here focus on the non-linear models introduced
in \cite{DJLS-spa}, and specify simple sufficient conditions implying the
assumptions used in the general results. The proofs of the results presented in
Sections~\ref{sec:main-results} and Section~\ref{sec:example-trawl} are
detailed in Section~\ref{sec:proofs}.  Before that, we introduce in
Section~\ref{sec:weight-weak-depend} some weighted weak dependence coefficients
that can be of independent interest but which will mainly serve us here to
check~\ref{item:hyp-consistency-common} for trawl processes.
Finally in Section~\ref{sec:numer-exper}, we present numerical experiments
focusing on the estimation of the long memory parameter $d^*$ comparing our
approach to the more classical local Whittle estimator, which is known to
perform well for standard linear models. Concluding remarks including
directions for future work are proposed in Section~\ref{sec:conclusion}. 

\section{Main results}\label{sec:main-results}
\subsection{Results on trawl processes}
As explained in the introduction, the $L^2$ convergence of~(\ref{g1}) follows
from~(\ref{eq:hyp-a1}).  We provide hereafter a more precise statement, and a
slight extension to a convergence in $\L^{2p}$ with $p\geq1$. All the proofs of
this section are postponed to Section~\ref{sec:proof-trawl-processes}.
\begin{lem}\label{lemme1}
  Assume~\ref{item:basic}. Then~(\ref{eq:hyp-a1}) implies that the
  convergence~(\ref{g1}) holds in $\L^2$ and the resulting process $X$ is
  ergodic. A centered version of $X$ can be obtained by setting
  \begin{align}
&\tilde{\gamma}(u)=\gamma(u)-\E\gamma(u)\;,    \\
    \label{eq:centered-trawl}
&  \tilde X_k=X_k-  \sum_{j=0}^\infty \E\,
    \gamma(a_j)=\sum_{j=0}^\infty \tilde{\gamma}_{k-j}(a_j)\;,\qquad
    k\in\Z\;.
  \end{align}
  If moreover, we have, for some $p>1$,
  \begin{equation}
    \label{eq:hyp-a1-p}
    \sum_{j=0}^\infty\left\|\tilde{\gamma}(a_j)\right\|_{2p}^{2p}<\infty\;,
  \end{equation}
  then the convergence~(\ref{g1}) also holds in $\L^{2p}$.
\end{lem}
Having a condition for $X$ to satisfy~\ref{item:hyp-consistency-ergo} and to be
$L^{2p}$, we now provides conditions for
obtaining~\ref{item:hyp-consistency-common}, which requires $p\geq2$ for the
considered covariances to be well defined.
  \begin{theo}\label{theo:trawl-gen-ass}
    Assume~\ref{item:basic} and~(\ref{eq:hyp-a1}). Then $X$ is an
    $\L^{2}$ stationary process (by
    Lemma~\ref{lemme1}), and~\ref{item:hyp-consistency-common} follows from
    any of
    the two following assertions with the same $s_0\in(0,1)$.
    \begin{enumerate}[label=(\roman*)]
    \item\label{item:trawl-gen-ass1} We have~(\ref{eq:hyp-a1-p}) with $p=2$ and , as
      $r\to\infty$,
      \begin{align}
    \label{eq:cond-weak-dep-consistency-commen-p3}
    &  \sum_{j=r}^\infty\v\,\gamma(a_j)= {\cal O}(r^{-2s_0})\;,\\        
            \label{eq:cond-weak-dep-consistency-commen-p2add}
 &\sum_{j=r}^\infty\left\|\tilde{\gamma}(a_j)\right\|_{2p}^{2p}= {\cal
  O}(r^{-2p\ s_0}) \;.      
      \end{align}
    \item\label{item:trawl-gen-ass2} We have~(\ref{eq:hyp-a1-p}) with $p=3$
      and~(\ref{eq:cond-weak-dep-consistency-commen-p3}) holds as $r\to\infty$.
\end{enumerate}
  \end{theo}
\subsection{Second order estimation}
\label{sec:second-order-estim}
In this section, we suppose that $X$ is a weakly stationary process with
auto-covariance $r$ or spectral density $f$. All the proof of this section are
postponed to Section~\ref{sec:postponed-proofs} for convenience.

The main assumption that we will require on $X$
is~\ref{item:hyp-consistency-common}.  It is interesting to note that,
if~(\ref{eq:cov-gen-ass-bound}) holds then
assuming~(\ref{eq:pairwise-cov-gen-ass-bound})
and~(\ref{eq:13wise-cov-gen-ass-bound}) is equivalent to assuming
\begin{equation}
  \label{eq:cum-gen-ass-bound}
  \left|\cum{X_{t_1}\ ,\ X_{t_2}\ ,\ X_{t_3}\ ,\ X_{t_4}}\right|\leq C_2\;\left[(1+t_4-t_3)^{-s_0}\wedge(1+t_3-t_2)^{-s_0}\right]\;.
\end{equation}
The precise statement is the following.
\begin{lem}\label{lem:cons-emp-cov-centered}
    Let $X$ be a weakly stationary process with zero mean such
    that~(\ref{eq:cov-gen-ass-bound}) holds for all $t_1\leq t_2$ in
    $\Z$. Then, for all  $t_1\leq t_2 \leq t_3\leq t_4$  in
    $\Z$,~(\ref{eq:pairwise-cov-gen-ass-bound})
    and~(\ref{eq:13wise-cov-gen-ass-bound}) imply (\ref{eq:cum-gen-ass-bound}),
    with $C_2=C_1+3C_0^2$ and (\ref{eq:cum-gen-ass-bound}) implies~(\ref{eq:pairwise-cov-gen-ass-bound})
    and~(\ref{eq:13wise-cov-gen-ass-bound}) with $C_1=C_2+3C_0^2$. 
\end{lem} 
We denote the empirical covariance function  by
\begin{equation}\label{eq:emp-cov-def}
  \widehat r_n(m)=\frac1{n}\sum_{j=1}^{n-m}(X_j-\bar X_n)(X_{j+m}-\bar X_n)
\end{equation}
where $\bar X_n$ denotes the empirical mean of the sample $X_1,\dots,X_n$.
The centering in the definitions of $\widehat r_n$ 
can be treated separately.  Define non-centered covariance estimator
\begin{equation}
  \label{eq:emp-cov-centered}
\widetilde r_n(k)= \frac1{n}\sum_{j=1}^{n-k}  X_j  X_{j+k}\;.
\end{equation}
The empirical covariance function defined by~(\ref{eq:emp-cov-def}) can then be
written as
\begin{equation}
  \label{eq:decomp-emp-cov}
\widehat r_n(k)=\widetilde r_n(k) - R_n^r(k)
\end{equation}
where $\widetilde r_n$ is the non-centered empirical covariance function
defined in~(\ref{eq:emp-cov-centered})
and
$R^n_n(k)$ is the reminder term defined by
\begin{equation}
  \label{eq:decomp-emp-cov-remainder}
R^r_n(k)=  \frac {k}n\left(\bar X_n\right)^2+\bar X_{n}\ \left(\frac1n\sum_{j=k+1}^{n-k} X_j\right)\;.
\end{equation}
This term is ``small'' only if $X$ is a centered process. Nevertheless, $X$ can
be assumed centered here, since the empirical covariance $\widehat r_n$ is
unchanged  when  $X$ is replaced by its centered version.

In the case where $X$ has mean zero, we have the following result. 
\begin{prop}\label{prop:cons-emp-cov-centered}
  Let $X$ be an $L^4$ process with zero mean and
  satisfying~\ref{item:hyp-consistency-common}. Then there exists a constant
  $C'$ only depending on $C_0,C_1$ and $s_0$ such that, for all $0\leq k\leq \ell<n$, 
 \begin{equation}
   \label{eq:empcov-cov-gen-ass}
\abs{\cov(\widetilde r_n(k)\ ,\ \widetilde r_n(\ell))}\leq C'\, n^{-s_0}\;.
 \end{equation}
\end{prop}
The following result follows.
\begin{cor}\label{cor:covar-estim}
  Let $X$ be a weakly stationary $L^4$ process 
  satisfying~\ref{item:hyp-consistency-common} with covariance function $r$.
Then there exists a constant
  $C'$ only depending on $C_0,C_1$ and $s_0$ such that, for all $0\leq k\leq \ell<n$, 
 \begin{equation}
   \label{eq:empcov-cov-noncentered}
\max_{0\leq k<n}\E\, \abs{\hat r_n(k)-r(k)}\leq C'\, n^{-s_0/2}\;.
 \end{equation}
\end{cor}
Another possible application of Proposition~\ref{prop:cons-emp-cov-centered} is
the pointwise Kernel estimation of the spectral density $f$ wherever it is well
defined and smooth. Let $J$ denotes a two times continuously differentiable
function with support $[-1/2,1/2]$ and such that $\int J=1$.
For any $\bandwidth>0$ and $\lambda_0\in[-\pi,\pi]$, let $J_{\bandwidth,\lambda_0}$ denotes its
$\lambda_0$-shifted, $\bandwidth$-scaled and $(2\pi)$-periodic version: 
$$
J_{\bandwidth,\lambda_0}(\lambda)=\frac1{\bandwidth}\sum_{k\in\Z}J\left(\frac{\lambda-\lambda_0-2\pi
    k}{\bandwidth}\right)\;.
$$
Define the Kernel estimator of $f(\lambda_0)$
$$
\hat f_{n,\bandwidth}(\lambda_0)=\int_{\tore}I_n \, J_{\bandwidth,\lambda_0}\;.
$$
Let $\mu$ denote the spectral measure of $X$ and suppose that it admits a
density $f$ in the neighborhood of $\lambda_0$, 
and that this density is continuous at $\lambda_0$. Then, it is easy to show that
\begin{equation}
  \label{eq:kernel-conv-bias}
\lim_{\bandwidth\to0}\int J_{\bandwidth,\lambda_0}\,\rmd\mu=f(\lambda_0)
\end{equation}
and the rate of convergence as $\bandwidth\to0$ can be obtained from the smoothness
index of $f$ at $\lambda_0$. This deterministic limit can be interpreted as a
control on the bias of the estimator  $\hat f_{n,\bandwidth_n}(\lambda_0)$ of
$f(\lambda_0)$. The deviation is bounded by the following result. 
\begin{cor}\label{cor:spdens-estim}
  Let $\mu$ be the spectral density associated to the covariance function
  $r$. Define $I_n$ and $\hat r_n$ by~(\ref{eq:periodo-def})
  and~(\ref{eq:emp-cov-def}).  Let $J$ be a kernel function as above and define
  the kernel estimator $\hat f_{n,\bandwidth}$ accordingly.
  Then~(\ref{eq:empcov-cov-noncentered}) implies that there exists a constant
  $C''$ only depending on $C'$, $r(0)$ and $J$ such that, for any $\lambda_0\in[-\pi,\pi]$,
\begin{equation}
  \label{eq:kernel-conv-dev}
  \E\,\abs{\hat f_{n,\bandwidth}(\lambda_0)-\int J_{h,\lambda_0}\,\rmd\mu}\leq C''\,n^{-s_0/2}\,\bandwidth^{-1}\;.
\end{equation}
\end{cor}
As usual, the deviation bound~(\ref{eq:kernel-conv-dev}) (where $\bandwidth=\bandwidth_n$
should not converge to 0 at a rate faster than $n^{s_0/2}$) has to be balanced with the
convergence~(\ref{eq:kernel-conv-bias}) (where $\bandwidth=\bandwidth_n$ should
converge to 0, the faster the better).

\subsection{Parametric Whittle estimation}
\label{sec:param-whittle-estim}
Although $h^*$ is an unknown element in the infinite dimensional space $\cont$,
our approach is parametric in nature in the sense that we now assume that
$(d^*,h^*)$ belongs to a known compact subset $K$ of $[0,1/2]\times\cont$. In
practice, to get a good approximation of an element of $\cont$, only a finite
number of its Fourier coefficients needs to be estimated. More generally we
denote by $(K_n)$ a sequence of subsets of $K$ in which we can always find
$(d^*,h_n^*)$ such that $h^*_n$ approximates $h^*$ well for $n$ large. More
precisely we consider the following assumption.
\begin{enumerate}[resume*=hyp]
\item\label{item:hyp-consistency} Let $K$ be a compact subset of
  $[0,1/2]\times\cont$ such that, for all $(d,h)\in K$, $f_d+h>0$ on
  $\rset$, and
  let $(K_n)$ be a sequence of subsets of $K$ such that for a well chosen
  sequence $(h^*_n)\in\cont^\N$, we have $(d^*,h^*_n)\in K_n$ for all $n\in\N$ and
  $h^*_n$ converges to $h^*$ uniformly.
\end{enumerate}
\begin{Rem}\label{rem:hn-Kn}
  If $h\in K$ is parameterized by
  finitely many parameters, one can take  $K_n=K$ for all $n\geq1$,
  in which case the last assertion of~\ref{item:hyp-consistency} is immediately
  satisfied for all $(d^*,h^*)\in K$ by taking $h^*_n=h^*$ for all $n$. 
\end{Rem}
An infinite dimensional setting can be set up as follows.  For any $s,C>0$, let
$H(s,C)$ denote the ball of even, real and locally integrable $(2\pi)$-periodic
functions $h:\R\to\R$ such that the Fourier coefficients
$$
c_k(h)=\intperarg{h(\lambda)\,\rme^{\rmi\lambda k}}{\lambda}
\quad\text{satisfy}\quad  
\left|c_k(h)\right|\leq C\,(1+|k|)^{-1-s}\;,\qquad k\in\Z\,.
$$
For any non-negative integer $m$, let moreover $\mathcal{P}_m$ denote the set
of even real trigonometric polynomials of degree at most $m$. For any locally
integrable $(2\pi)$-periodic function $h$, denote by $p_m[h]$ the
projection of $h$ onto $\mathcal{P}_m$, that is,
$$
p_m[h](\lambda)=c_0(h)+\sum_{k=1}^m2\,c_k(h)\,\cos(\lambda\,k)\;,\quad \lambda\in\rset\;.
$$
For any $s,C>0$, it is easy to show that $\sup|h-P_m[h]|={\cal O}(m^{-s})$
uniformly in $h\in H(s,C)$ as $m\to\infty$.  The following result can be used
to build a parameter space $K$ and a sequence $(K_n)$
satisfying~\ref{item:hyp-consistency} from a given set
$A\subset[0,1/2]\times H(s,C)$ of couples $(d,h)$ containing the true
parameters.

\begin{lem}\label{lem:Kn-infinite-dim}
  Let $s,C>0$ and $A\subseteq[0,1/2]\times H(s,C)$ such that $f_d+h>0$ on
  $\rset$ for all $(d,h)\in A$. Suppose that $A$ is closed
  in $[0,1/2]\times\cont$ and let $(d^*,h^*)\in A$. Then there exists a
  positive integer $m_0$ such that $f_d+p_m[h]>0$ on
  $\rset$ for all $(d,h)\in A$ and $m\geq m_0$. Moreover, for any diverging
  sequence $(m_n)$ of integers larger than or equal to $m_0$,
  Assumption~\ref{item:hyp-consistency} holds 
  by setting $K_n=\{(d,p_{m_n}[h])~:~(d,h)\in A\}$, for all $n\geq1$ and
  $K=A\cup(\bigcup_n K_n)$. 
\end{lem}

The proof of this lemma is postponed to Section~\ref{sec:cons-param-whittle}.
We can now state the consistency of our estimator which, in the same flavor as
in \cite[Theorem~8.2.1]{gir2012}, only requires the observed process to be
ergodic. Its proof is also postponed to Section~\ref{sec:cons-param-whittle}. 
\begin{theo}\label{theo:cons}
  Suppose that the process $X$ satisfies~\ref{item:hyp-consistency-ergo} and admits a spectral density
  of the form~(\ref{eq:sp-dens-general}) with parameter $(c^*,d^*,h^*)$
  satisfying~\ref{item:hyp-consistency} for some subsets $K$ and $(K_n)$ 
  of $[0,1/2]\times\cont$.
  
  Let $(\hat d_{n},\hat h_{n})\in K_n$ such that, a.s., as $n\to\infty$,
  \begin{equation}
    \label{eq:def-param-edstim}
  \Lambda_n(\hat d_{n},\hat h_{n}) \leq \inf_{(d,h)\in K_n}  \Lambda_n(d,h)+o(1)\;,
\end{equation}
where $\Lambda_n$ is defined by~(\ref{eq:whittle-contrast}),
 and define $\hat c_{n}$ by~(\ref{eq:c-parm-estim}). Then, a.s.,  $\hat d_{n}$
 and $\hat c_{n}$ converge to $d^*$ and $c^*$, and $\hat h_{n}$ converges to
 $h^*$ uniformly. 
\end{theo}
Assumption~\ref{item:hyp-consistency} provides a new framework of parametric
models, different from the ones classically used in Whittle parameter
estimation, and which seems to be well adapted for many examples of trawl
processes, see Section~\ref{sec:example-trawl}. However it also includes
many known cases. Let us examine the celebrated ARFIMA model, in which the
spectral density takes the form
\begin{equation}
  \label{eq:arfima-dens}
f(\lambda)=\sigma_*^2\,f_{d^*}(\lambda)\,\left|\frac{1+\sum_{k=1}^p\theta^*_k\rme^{-\rmi k\lambda}}{1-\sum_{k=1}^q\phi^*_k\rme^{-\rmi k\lambda}}\right|^2\;,  
\end{equation}
where, for some positive integers
$p$ and $q$,  the MA and AR coefficients
$\theta^*=(\theta^*_1,\dots,\theta^*_p)$ and $\phi^*=(\phi^*_1,\dots,\phi^*_q)$ are assumed to make the corresponding
ARMA process canonical. In the following this will be denoted by
$(\phi,\theta)\in\Theta_{p,q}$, defined by
\begin{multline}
  \label{eq:ThetaCanonical}
\Theta_{p,q}=\left\{(\phi,\theta)\in\rset^{p+q}~:~\text{$\Phi$ and $\Theta$
    have no common roots}\right.\\
\left.    \text{ and for all $z\in\cset$ such that $|z|\leq1$,}\Phi(z)\neq0\text{ and } \Theta(z)\neq0\right\}  \;,
\end{multline}
where $\Phi$ and $\Theta$ are the 
AR and MA polynomials defined by 
$$
\Phi(z):=1-\sum_{k=1}^p\phi_kz^k\text{ and }
\Theta(z):=1+\sum_{k=1}^q\theta_kz^k\;.
$$
The corresponding reduced Whittle contrast reads
\begin{equation}
  \label{eq:whittle-contrast-arfima}
\widetilde\Lambda_n(d,(\phi,\theta)) =
\ln\left(\intperarg{\frac{I_n(\lambda)}
  {f_d(\lambda)}\;\frac{|\Phi(\rme^{-\rmi\lambda})|^2}{|\Theta(\rme^{-\rmi\lambda})|^2}}{\lambda} \right)
\; ,
\end{equation}
where $f_{d}$ is defined by~(\ref{eq:arfima}) and $\intpermu$ is the Lebesgue
measure on $[-\pi,\pi]$ divided by $2\pi$.  The form~(\ref{eq:arfima-dens}) is
in fact a special case of~(\ref{eq:sp-dens-general}) by setting
\begin{align}
  \label{eq:hdtheta-arfima}
h^*(\lambda)&=f_{d^*}(\lambda)\left(\left|\frac{\Theta(\rme^{-\rmi\lambda})\,\Phi(1)}{\Phi(\rme^{-\rmi\lambda})\,\Theta(1)}\right|^2-1\right)\\
    \label{eq:csigma2-arfima}
  c^*&=\sigma_*^2\,\left|\frac{\,\Theta(1)}{\Phi(1)}\right|^2\;.
\end{align}
Note that $h^*$ is indeed continuous.
The ARFIMA linear processes have been
extensively studied. However the usual proof of the consistency relies on the
Hannan's approach of \cite{hannan1973} but it does not hold if $d=0$ is
included in the set of parameters. 
Here, as a consequence of
Theorem~\ref{theo:cons}, we get the following, which provides an
alternative proof.
\begin{cor}\label{cor:arfima}
Let $p,q$ be
  two positive integers and $\tilde K$ be a compact subset of
  $[0,1/2)\times\Theta_{p,q}$.   Suppose that the process $X$ satisfies~\ref{item:hyp-consistency-ergo} and
  admits a spectral density of the form~(\ref{eq:arfima-dens}), with
  $(d^*,(\phi^*,\theta^*))\in\tilde K$ and $\sigma_*>0$. 
  
  Let $(\hat d_n,\hat\vartheta_n)\in\tilde K$ such that, a.s., as $n\to\infty$,
  \begin{equation}
    \label{eq:def-param-edstim-arfima}
  \widetilde\Lambda_n(\hat d_n,\hat\vartheta_n)\leq\inf_{(d,\vartheta)\in\tilde
  K}  \widetilde\Lambda_n(d,\vartheta)+o(1)\;,    
  \end{equation}
  where $\widetilde\Lambda_n$ is defined by~(\ref{eq:whittle-contrast-arfima}).
  Define moreover
\begin{align}
    \label{eq:sigma-parm-estim-arfima}
\hat\sigma^2_n=\exp\left(\widetilde\Lambda_n(\hat d_n,\hat\vartheta_n) \right)\;.
\end{align}
Then, a.s.,  $\hat d_{n}$, $\hat \vartheta_{n}$ and $\hat\sigma^2_n$ converge to $d^*$,
 $\vartheta^*=(\phi^*,\theta^*)$ and $\sigma_*^2$.
\end{cor}
\begin{proof}
See Section~\ref{sec:arfima-case:-proof}.
\end{proof}

\section{Examples of discrete time trawl processes}
\label{sec:example-trawl}
\subsection{Random line seed}

As explained in \cite[Example~1]{DJLS-spa}, any causal linear process is a trawl
process by setting the seed process to be the random line seed
$\gamma(t)=t\epsilon$, where $\epsilon$ is a random variable with zero mean and
finite variance.

The parametric estimation in the linear case is a well known topic, usually
treated using ARFIMA parametrization, see e.g. \cite[Section~8.3.2]{gir2012} for a
complete statistical analysis of this model.

\subsection{Lévy seed and non-increasing sequence}
\label{sec:examples}

Consider the two following assumptions
\begin{enumerate}[resume*=hyp]
\item\label{item:hyp-gamma} The process $\gamma$ is a L{\'e}vy process with finite
  variance normalized so that $ \v \,\gamma(1)=1$.
\item \label{item:hyp-a} The sequence $a$ is non-increasing and there exist $
  c > 0$, and $\alpha^*>1$ such that~(\ref{g4}) holds.
\end{enumerate}
They imply~(\ref{eq:hyp-a1}) since then we have, for all $t\geq0$, $\E\,\gamma(t)=\delta\,t$ for some drift
$\delta$ and  $ \v \,\gamma(t)=t$. By Lemma~\ref{lemme1} and
Eq.~(\ref{eq:covariance-trawl}), the trawl process $X$ defined by~(\ref{g1})
satisfies~\ref{item:hyp-consistency-ergo} and its auto-covariance function
  $r$ is given by
\begin{equation} \label{rk-exact-formula}
  r(k)=\sum_{j\geq k}a_j\;,\qquad k\in\nset\;.
\end{equation}
If \ref{item:hyp-gamma}
and~\ref{item:hyp-a} hold and $\gamma(1)$ admits a finite
$q$-th moment, we easily have that, for all $t_1\leq \dots\leq t_q$ in $\Z$,
\begin{align*}
\cum{X_{t_1}\ ,\ \dots\ ,\
  X_{t_q}}&=\sum_{j\geq0}\cum{\gamma(a_{t_k-t_1+j}),\
            k=1,\dots,q}\\
          &=\kappa_q\;\sum_{k\geq t_q-t_1}a_{k}\;,  
\end{align*}
where $\kappa_q$ is the $q$-th order cumulant of $\gamma(1)$. 
We then obtain 
\begin{align*}
\cum{X_{t_1}\ ,\ X_{t_2}\ ,\ X_{t_3}\ ,\
  X_{t_4}}&={\cal O}((t_4-t_1)^{1-\alpha^*})\;.
\end{align*}
So, by Lemma~\ref{lem:cons-emp-cov-centered}, if $q=4$, $X$
satisfies~\ref{item:hyp-consistency-common} with $s_0=\alpha^*-1$.
Theorem~\ref{theo:trawl-gen-ass} shows that
Condition~\ref{item:hyp-consistency-common} continues to hold for more general
trawl processes, provided some adequate moment conditions, but with $s_0$
possibly higher than $\alpha^*-1$ (see Section~\ref{sec:other-seeds} for
examples).

For such a process, we can specify $(a_k)$ so that the spectral density is of
the form~(\ref{eq:sp-dens-general}) with $(d^*,h^*)$ lying within a parameter
space $K$ satisfying Condition~\ref{item:hyp-consistency}.
A very special case, detailed in \cite[Example~5]{DJLS-spa}, consists in setting
\begin{equation}
  \label{eq:farim-seq-ak}
a_k=c^*\,\left(r_k^{(d^*)}-r_{k+1}^{(d^*)}\right)\,,\qquad k\in\N\;,
\end{equation}
where, for all $d<1/2$,  $r^{(d)}$ is defined as the auto-covariance function of ARFIMA$(0,d,0)$ with unit variance
innovation, that is,
\begin{equation}
  \label{eq:farim-cov}
r^{(d)}(k)=\intperarg{\left|1-\rme^{-\rmi\lambda}\right|^{-2d}\rme^{\rmi\lambda k}}{\lambda}\,,\qquad k\in\Z\;.  
\end{equation}
It is shown in \cite{DJLS-spa} that, for any $d^*\in(0,1/2)$ such a sequence $(a_j)$
satisfies~\ref{item:hyp-a} with $\alpha^*=2(1-d^*)\in(1,2)$, so that, under~\ref{item:hyp-gamma},
following~(\ref{rk-exact-formula}) and~(\ref{eq:farim-cov}), the corresponding
trawl process has a spectral density of the form~(\ref{eq:sp-dens-general})
with $h^*=0$.

We check in the following section that more general seed processes and trawl
sequences can be used.

\subsection{More general seeds and sequences}
\label{sec:other-seeds}

In this section, in contrast to~\ref{item:hyp-a}, we consider trawl sequences $(a_j)$ that may not be
non-increasing but we specify~(\ref{g4}) by assuming that, there exists $c>0$
and $\alpha^*\in(1,2)$ such that
\begin{equation} 
  \label{eq:g4-bis}
0\leq    a_j \ = \  c \ j^{-\alpha^*} (1+ O(j^{-1}))\quad\text{as}\quad j \to \infty  \;.
\end{equation}
We also consider the non Lévy seed processes introduced in
\cite{DJLS-spa}, for which the covariance structure can still be derived precisely.
Let us examine here the mixed Poisson seed and the Binomial seed processes of
their Examples 3 and 4. The first case extends the (thus Lévy) Poisson seed by
setting $\gamma(t)=N(\zeta \,t)$, where $N$ is a homogeneous Poisson counting
process with unit intensity and $\zeta$ is a positive random variable
independent of $N$ and with finite variance. Then we have, for all $u,v\geq0$,
$\E\,\gamma(u)=u\,\E\zeta$ and
$\cov(\gamma(u),\gamma(v))=(u\wedge v)\E\zeta+uv\var(\zeta)$.
Thus, for any
sequence $(a_j)$ satisfying~(\ref{eq:g4-bis}), Condition~(\ref{eq:hyp-a1})
holds and~\ref{item:hyp-consistency-ergo} follows from Lemma~\ref{lemme1} and
Eq.~(\ref{eq:covariance-trawl}) yields the following auto-covariance function
for $X$~:
$$
r(k)=\E\zeta\,\sum_{j=0}^\infty (a_j\wedge
a_{j+k})+\var(\zeta)\sum_{j=0}^\infty a_ja_{j+k}\;,\qquad k\in\N\;.
$$
If moreover $\E\zeta^6<\infty$, then~(\ref{eq:hyp-a1-p}) holds with $p=3$, and,
by Theorem~\ref{theo:trawl-gen-ass}~\ref{item:trawl-gen-ass2}, we
get~\ref{item:hyp-consistency-common} with $s_0=(\alpha^*-1)/2$. If we only
assume that $\E\zeta^4<\infty$,
then~(\ref{eq:cond-weak-dep-consistency-commen-p2add}) holds with $p=2$ and
$s_0=(\alpha^*-1)/4$, so that
Theorem~\ref{theo:trawl-gen-ass}~\ref{item:trawl-gen-ass1} gives
that~\ref{item:hyp-consistency-common} holds this time only with
$s_0=(\alpha^*-1)/4$.

The Binomial seed process of \cite[Example~4]{DJLS-spa} is defined for some given
$n\in\nset^*$ by setting $\gamma(t)=\sum_{i=1}^n\1_{\{U_i\leq t\}}$ with the
$U_i$'s i.i.d. and uniform on $[0,1]$. In this case, we have that, for all
$u\geq1$, $\gamma(u)=n$ and, for all $u,v\in[0,1]$, $\E\,\gamma(u)=n\,u$ and
$\cov(\gamma(u),\gamma(v))=n(u\wedge v-uv)$. Thus, for any sequence $(a_j)$
satisfying~(\ref{eq:g4-bis}), similarly to the previous
case,~\ref{item:hyp-consistency-ergo} holds and the trawl process $X$ has
auto-covariance function $r$ given by
$$
r(k)=n\,\sum_{j=0}^\infty (\tilde a_j\wedge
\tilde a_{j+k})-n\,\sum_{j=0}^\infty \tilde a_j\tilde a_{j+k}\;,\qquad k\in\N\;,
$$
where, for all $j\in\N$, $\tilde a_j=a_j\1_{\{a_j<1\}}$. Also, for the binomial
seed and $(a_j)$ satisfying~(\ref{eq:g4-bis}),~(\ref{eq:hyp-a1-p}) holds for any integer
$p$, and~\ref{item:hyp-consistency-common} holds with $s_0=(\alpha^*-1)/2$ by 
Theorem~\ref{theo:trawl-gen-ass}~\ref{item:trawl-gen-ass2}.

Having checked that the trawl process satisfies~\ref{item:hyp-consistency-ergo}
and~\ref{item:hyp-consistency-common} for these seeds, we now turn to the form
of its spectral density and show that it is indeed of the
form~(\ref{eq:sp-dens-general}) and can be used with
Lemma~\ref{lem:Kn-infinite-dim} to form a parameter space $K$ that
satisfies~\ref{item:hyp-consistency}.
  
\begin{prop}\label{prop:other-seeds}
  Assume~\ref{item:basic}. Suppose that $\gamma$ is Lévy seed process, a mixed
  Poisson seed process or a binomial seed process. Suppose moreover that
  $\gamma(1)$ has finite positive variance and $(a_j)$
  satisfies~(\ref{eq:g4-bis}) with $\alpha^*\in(1,2)$. Then the trawl process
  defined by~(\ref{g1}) has a spectral density of the
  form~(\ref{eq:sp-dens-general}) with $d^*=1-\alpha^*/2\in(0,1/2)$ and 
  $h^*\in H(\alpha^*-1,C)$ for some $C>0$.
\end{prop}
\begin{proof}
  See Section~\ref{sec:proof-trawl-processes}.
\end{proof}

\section{Weighted weak dependence indices}
\label{sec:weight-weak-depend}

Here we introduce a somewhat general setting that will be used later to derive
some important properties on the memory of Trawl processes. They can be,
however, of independent interest.

We use the classical weak-dependence concept.
\begin{Def} [\cite{Dedecker2007}]\label{defwd}
 A random process $(X_t)_{t\in \Z}$ is said to be $\theta-$weakly dependent if
\begin{equation}\label{depf}
\left|\cov \left(f(X_{i_1},\ldots,X_{i_u}),g(X_{j_1},\ldots,X_{j_v})\right)\right|\le
\theta_r\, v\;,
\end{equation}
for  $i_1\le\cdots\le  i_u\le j_1-r\le j_1\le\cdots\le j_v$ and functions $f:\R^u\to\R$ with $\|f\|_\infty\le 1$ and $g:\R^v\to\R$ with
$$|g(y_1,\ldots,y_v)-g(x_1,\ldots,x_v)|\leq |y_1-x_1|+\cdots+|y_v-x_v|\;. $$
\end{Def}
\begin{Def}
  A time series $(X_k)$ is said to be a causal Bernoulli shift process (CBS) if
  there exists an iid sequence $(\gamma_j)_{j\in\Z}$ valued in
  $(E,\mathcal{E})$ and a measurable function $\Phi:E^\nset\to\R$ such that,
  for all $k\in\Z$, $X_k=\Phi((\gamma_{k-j})_{j\geq0})$. The $L^q$ coefficients
  $(\pi^{(q)}_r)_{r\geq1}$ of $(X_k)$ are then defined
  by
    \begin{align}
    \label{eq:couplingargument}
    \pi^{(q)}_r&=\left\|\Phi((\gamma_j)_{j\geq0})-\Phi((\gamma_j)_{0\leq
              j<r},(\gamma'_j)_{j\geq r})\right\|_q\;,
    \end{align}
where $(\gamma'_j)_{j\in\Z}$ is an
  independent copy of $(\gamma_j)_{j\in\Z}$.
\end{Def}
Provided that a CBS process is well defined in $\L^2$, it is weakly dependent.
\begin{lem}\label{lemme1_theta-weak}
Let $X$ be an  $\L^{2}$ centered CBS process. Then it is $\pi^{(2)}-$weakly dependent. 
\end{lem}
\begin{proof}
We write $X_k=X'_k+(X_k-X'_k)$ where $(X'_k)$ is defined  by
$X'_k=\Phi((\gamma_{k-j})_{0\leq j<r},(\gamma'_{k-j})_{j\geq r}))$.  
Observe now that $X'_k$ is independent of $\sigma(\gamma_i,\,i\leq k-r)$,
hence of $\sigma(X_i,\,i\leq k-r)$. On the other hand we have that
$$
\E (X_k-X'_k)^2=\pi^{(2)}_r\;.
$$
Now take $f:\R^u\to\R$ with $\|f\|_\infty\le 1$ and $g:\R^v\to\R$ Lipschitz and
$i_1\le\cdots\le  i_u\le j_1-r\le j_1\le\cdots\le j_v$. Denoting
$$
m=\E\,f(X_{i_1},\ldots,X_{i_u})\;, 
$$
we get that $ \left|\cov
  \left(f(X_{i_1},\ldots,X_{i_u}),g(X_{j_1},\ldots,X_{j_v})\right)\right|$ is
bounded from above by
$$
\E\left[\left|f(X_{i_1},\ldots,X_{i_u})-m\right|\,\left(\sum_{k=1}^v|X_{j_k}-X'_{j_k}|\right)\right]\;.
$$
And we conclude with the Cauchy-Schwartz inequality.
\end{proof}
Using the same proof we can include polynomial terms in the functions $f$
and $g$.
\begin{Def} \label{defwdweighted}
 Let $\past{p},\fut{p}\geq1$. A random process $(X_t)_{t\in \Z}$ is said to be $(\past{p},\fut{p})$-weighted $\theta-$weakly dependent if
\begin{equation}\label{depf-weighted}
\left|\cov \left(f(X_{i_1},\ldots,X_{i_u}),g(X_{j_1},\ldots,X_{j_v})\right)\right|\le
\theta_r\, v \;,
\end{equation}
for all $i_1\le\cdots\le  i_u\le j_1-r\le j_1\le\cdots\le j_v$, all 
 functions $f:\R^u\to\R$ satisfying
$$
\left|f(x_{i_1},\ldots,x_{i_u})\right|\leq \frac1{1+u}\left( 1+\sum_{k=1}^u|x_{i_k}|\right)^{\past{p}}\;,
$$
and  all functions $g:\R^v\to\R$ satisfying
$$
\frac{|g(y_1,\ldots,y_v)-g(x_1,\ldots,x_v)|\;(1+2v)}{\left(1+|y_1|+\cdots+|y_v|+|x_1|+\cdots+|x_v|\right)^{\fut{p}-1}}\leq 
\left(|y_1-x_1|+\cdots+|y_v-x_v|\right)\;.
$$ 
\end{Def}
\begin{Rem}\label{rem:weighted-weak-depend}
  Note that in Definition~\ref{defwdweighted}, the conditions on $f$ and $g$
  are weaker as $\futpast{p}$ increases. Namely, if $1\leq \futpast{p}\leq \futpast{p'}$, then a $(\past{p'},\fut{p'})$-weighted
  $\theta-$weakly dependent random process $(X_t)_{t\in \Z}$ is also
  $(\past{p},\fut{p})$-weighted $\theta-$weakly dependent.
\end{Rem}
Using this new definition, we get the following result.
\begin{lem}\label{lemme1weights}
Let $(X_k)$ be an  $\L^{2p}$ centered CBS process. Then, for any $\futpast{p}\geq1$
such that $\past{p}+\fut{p}=2p$,
it is $(\past{p},\fut{p})$-weighted $\theta$-weakly
  dependent with
$$
\theta_r\leq \left(1\vee\|X_0\|^{2p-1}_{2p} \right)\; \pi^{(2p)}_r\;.
$$
\end{lem}
\begin{proof}
  Let us now prove the bound of the  $p$-weighted $\theta-$weak dependence
  coefficient $\theta^{(p)}_r$.
  We use the same notation as in the proof of Lemma~\ref{lemme1_theta-weak} but this time
  with $f$ and $g$ as in Definition~\ref{defwdweighted}.
  We then obtain that
  $$
  \left|\cov
  \left(f( X_{i_1},\ldots, X_{i_u}),g( X_{j_1},\ldots, X_{j_v})\right)\right|
$$
is bounded from above by
\begin{equation}
  \label{eq:wigth-mixing-bound-before-holder}
\frac1{1+2v}\,\E\left[\left|f( X_{i_1},\ldots, X_{i_u})-m\right|\,\left(\sum_{k=1}^v|X_{j_k}-X'_{j_k}|\right)\,\left(1+\sum_{k=1}^v| X_{j_k}|+|X'_{j_k}|\right)^{\fut{p}-1}\right]\;.
\end{equation}
Using the H{\"o}lder inequality with $\past{p}/(2p)+1/(2p)+(\fut{p}-1)/2p=1$, we obtain
\begin{equation}
  \label{eq:wigth-mixing-bound}
  \left|\cov
  \left(f( X_{i_1},\ldots, X_{i_u}),g( X_{j_1},\ldots, X_{j_v})\right)\right|\leq A\;B\;C\;,
\end{equation}
with
\begin{align*}
  A&=\left\|f( X_{i_1},\ldots, X_{i_u})-m\right\|_{2}
     \leq \left\|f( X_{i_1},\ldots, X_{i_u})\right\|_{2p/\past{p}}\\
  B&=\left\|\sum_{k=1}^v|X_{j_k}-X'_{j_k}|\right\|_{2p}\leq v\,\pi_r^{(2p)}\\
  C&=\frac
     1{1+2v}\,\left\|\left(1+\sum_{k=1}^v\left(| X_{j_k}|+|X'_{j_k}|\right)\right)^{\fut{p}-1}\right\|_{2p/(\fut{p}-1)}
    \;.
\end{align*}
We immediately have that, by the assumption on $f$  that
$$
A\leq\frac 1{1+u}\,\left(1+\sum_{k=1}^u
  \left\| X_{i_k}\right\|_{2p}\right)^{\past{p}}\leq1\vee
  \left\| X_{0}\right\|_{2p}^{\past{p}}\;.
$$
Finally, we note that
$C\leq
\left(1\vee\left\| X_0\right\|_{2p}^{\fut{p}-1}\right)$.
The result follows from~(\ref{eq:wigth-mixing-bound}) and the above bounds of
$A,B$ and $C$.
\end{proof}

We also obtained this lemma with an improved weighted weakly dependent
coefficient by conceding a bit of moment condition.

\begin{lem}\label{lemme1weights-better}
Let $(X_k)$ be an  $\L^{2p}$ centered CBS process. Then, for any $\futpast{p}\geq1$
  \begin{equation}
    \label{eq:cond-p-futpast-p}
  \past{p}+\fut{p}=p+1\;.    
  \end{equation}
it is $(\past{p},\fut{p})$-weighted $\theta$-weakly
  dependent with
$$
\theta_r\leq2\,\left(1\vee\left\| X_0\right\|_{2p}^{p}\right)\,\pi_r^{(2)}\;.
$$
where $C_p$ is a positive constant only depending on $p$ and $S_0(2p)$ is defined in~(\ref{eq:Sr2p}).
\end{lem}
\begin{proof}
  We use again the upper bound~(\ref{eq:wigth-mixing-bound-before-holder}) of
  $$
  \left|\cov
  \left(f( X_{i_1},\ldots, X_{i_u}),g( X_{j_1},\ldots, X_{j_v})\right)\right|\;,
$$
but we apply the H{\"o}lder inequality with the weights 
$\past{p}/(2p)+1/2+(\fut{p}-1)/(2p)=1$ (which holds by~(\ref{eq:cond-p-futpast-p})) and obtain that
\begin{equation}
  \label{eq:wigth-mixing-bound-new}
  \left|\cov
  \left(f( X_{i_1},\ldots, X_{i_u}),g( X_{j_1},\ldots, X_{j_v})\right)\right|\leq A'\;B'\;C'\;,
\end{equation}
with
\begin{align*}
  A'&=\left\|f( X_{i_1},\ldots, X_{i_u})-m\right\|_{2p/\past{p}}\leq
      2\ \left\|f( X_{i_1},\ldots, X_{i_u})\right\|_{2p/\past{p}}\\
  B'&=\left\|\sum_{k=1}^v|X_{j_k}-X'_{j_k}|\right\|_{2}\leq
      v\,\pi_{r}^{(2)}\\
  C'&=\frac
      1{1+2v}\,\left\|\left(1+\sum_{k=1}^v\left(| X_{j_k}|+|X'_{j_k}|\right)\right)^{\fut{p}-1}\right\|_{2p/(\fut{p}-1)}
      \;.
\end{align*}
We immediately have that, by the assumption on $f$  that
$$
A'\leq\frac 2{1+u}\,\left(1+\sum_{k=1}^u
  \left\| X_{i_k}\right\|_{2p}\right)^{\past{p}}\leq 2\left(1\vee
  \left\| X_{0}\right\|_{2p}^{\past{p}}\right)\;.
$$
Finally, we note that, similarly,
$C'\leq
\left(1\vee\left\| X_0\right\|_{2p}^{\fut{p}-1}\right)$.
The result follows from~(\ref{eq:wigth-mixing-bound-new}) and the above bounds of
$A',B'$ and $C'$.
\end{proof}

\section{Proofs}\label{sec:proofs}
\subsection{On trawl processes}
\label{sec:proof-trawl-processes}\begin{proof}[Proof of Lemma~\ref{lemme1}]
  We prove the result under~(\ref{eq:hyp-a1}) and~(\ref{eq:hyp-a1-p}). The case
  where~(\ref{eq:hyp-a1-p}) is not assumed corresponds to setting $p=1$ in the
  following. 
  By the Rosenthal Inequality for sums of independent random
  variables, see \cite[Theorem~2.9]{petrov95}, we have, for any
  $1\leq i\leq k$, for some constant $C_p$ only depending on $p$,
  \begin{equation}
    \label{eq:rosenthal-p-trawl}
    \left\|\sum_{j=i}^k\gamma_j(a_j)\right\|_{2p}\leq   C_p\,
\sum_{j=i}^k \left|\E\,
  \gamma(a_j)\right|+
\left(\sum_{j=i}^k\left\|\tilde{\gamma}(a_j)\right\|_{2p}^{2p}\right)^{1/2p}
+\left(\sum_{j=i}^k\v\,\gamma(a_j)\right)^{1/2}
  \end{equation}
The convergence of~(\ref{g1}) in $\L^{2p}$ follows.

It follows that we can write $X_k$ as $X_k=F((\gamma_{k-j})_{j\geq0})$ with $F$
measurable from $\R^{\R}$ to $\R$, with $\R^{\R}$ endowed by the $\sigma$-field
$\mathcal{B}(\R)^{\otimes\R}$ (the smallest one that makes the $\R^\R\to\R$
mapping $x\mapsto x(t)$ measurable for all $t\in\R$). Since
$(\gamma_{j})_{j\in\Z}$ is i.i.d., it is ergodic, and so is $(X_k)_{k\in\Z}$.

All the other assertions of the lemma are obvious.
\end{proof}
Using the same idea and the results of Section~\ref{sec:weight-weak-depend}, we
now prove  Theorem~\ref{theo:trawl-gen-ass}.
\begin{proof}[Proof of Theorem~\ref{theo:trawl-gen-ass}]
  Let us  now prove Theorem~\ref{theo:trawl-gen-ass}. Note
  that
  $$
  r(k)=\sum_{j\geq 0}\cov \Big(\gamma(a_j),\gamma(a_{k+j})\Big)\leq
  \left(\sum_{j\geq 0} \v\ \gamma(a_j)\right)^{1/2}  \left(\sum_{j\geq k} \v\ \gamma(a_{j})\right)^{1/2}\;.
  $$
  Hence~(\ref{eq:cond-weak-dep-consistency-commen-p3}) implies~(\ref{eq:cov-gen-ass-bound}).

  It remains to show~(\ref{eq:pairwise-cov-gen-ass-bound})
  and~(\ref{eq:13wise-cov-gen-ass-bound}). We use that $(\tilde X_k)$ defined
  in~(\ref{eq:centered-trawl}) can be written the causal Bernoulli shift
  process 
  $$
  \tilde
  X_k=\Phi((\gamma_{k-j})_{j\geq0})\quad\text{with}\quad\Phi((\gamma_j)_{j\geq0})=\sum_{j=0}^\infty\tilde{\gamma}_j(a_j)\;. 
  $$
  where $\Phi$ is a measurable mapping on $E^\nset$, with $E=\R^{\R}$ endowed
  with $\mathcal{B}(\R)^{\otimes\R}$.  Then the $L^q$ coefficients defined in~(\ref{eq:couplingargument})
  with  $(\gamma'_j)_{j\geq0}$ denoting an
  independent copy of $(\gamma_j)_{j\geq0}$, satisfy, for all $r\in\nset$,
  and $q\geq2$,
  \begin{align*}
    \pi^{(q)}_r
    &=     \left\|
      \sum_{j=r}^\infty\left(\gamma_j(a_j)-\gamma'_j(a_j)\right)\right\|_q\\
      &\leq 2\,C_qS_r(q)\;,
  \end{align*}
  where the second inequality follows from~(\ref{eq:rosenthal-p-trawl}) by setting
\begin{equation}
  \label{eq:Sr2p}
S_r(q)=
\left(\sum_{j=r}^\infty\left\|\tilde{\gamma}(a_j)\right\|_{q}^{q}\right)^{1/q}
+\left(\sum_{j=r}^\infty\v\,\gamma(a_j)\right)^{1/2}\;.   
\end{equation}
  
  We now separate the two cases. 
  
  In the case where $p=2$ and
  both~(\ref{eq:cond-weak-dep-consistency-commen-p3})
  and~(\ref{eq:cond-weak-dep-consistency-commen-p2add}) hold, we apply
  Lemma~\ref{lemme1weights} with $\past{p}=\fut{p}=2$ and $\past{p}=3$,
  $\fut{p}=1$, successively. This gives~(\ref{eq:pairwise-cov-gen-ass-bound})
  and~(\ref{eq:13wise-cov-gen-ass-bound}), respectively.

  In the case where $p=3$, we only need~(\ref{eq:cond-weak-dep-consistency-commen-p3}) to hold, as we can
  apply Lemma~\ref{lemme1weights-better} with $\past{p}=\fut{p}=2$ and
  $\past{p}=3$, $\fut{p}=1$, successively. 
\end{proof}

The following lemma is useful for proving Proposition~\ref{prop:other-seeds}. 
\begin{lem}\label{lem:sum-wedge-lemma}
  Let $\alpha>1$. Let $(b_k)$ be a non-negative sequence such that
  $b_k=(1+k)^{-\alpha}\,(1+O(k^{-1}))$ as $k\to\infty$. Then we have, as $k\to\infty$,
  $$
  \sum_{j\geq0} (b_j\wedge b_{j+k})= \sum_{j\geq k} b_j +O(k^{-\alpha})\;. 
  $$
\end{lem}
\begin{proof}
  First observe that, for all $k\in\nset$, 
  \begin{equation}
    \label{eq:ub-sum-wedge-lemma}
      \sum_{j\geq0} (b_j\wedge b_{j+k})\leq \sum_{j\geq0} b_{j+k} \;.
    \end{equation}
    Now, there exists $C>0$ such that for all $j\in\N$,
  \begin{equation}
    \label{eq:ub-lb-interm-sum-wedge-lemma}
         (1+j)^{-\alpha}\,(1-C\,(j+1)^{-1})\leq b_j\leq      (1+j)^{-\alpha}\,(1+C\,(j+1)^{-1})\;.
       \end{equation}
       It
    follows from the first inequality that, for all $j\in\N$ and $k\geq (1+C)/\alpha$, 
    \begin{align*}
      \frac{b_j}{(j+k+1)^{-\alpha}}
                                 &  \geq (1-C\,(j+1)^{-1})\,(1+\alpha^{-1}(1+C)(j+1)^{-1})^{\alpha} \\
      &= 1+(j+1)^{-1}+O(j^{-2})      \quad\text{as $j\to\infty$.}
    \end{align*}
In particular, the latter term is larger than or equal to $1$ for $j$ large
enough and it follows that there exists $j_0$ only depending on $\alpha$ and
$C$ such that, for all $j\geq j_0$ and $k\geq  (1+C)/\alpha$,
$$
b_j\geq(j+k+1)^{-\alpha}\geq b_{j+k}-C\,(j+k+1)^{-1}\;,
$$
where we used the second inequality of~(\ref{eq:ub-lb-interm-sum-wedge-lemma}).
This now implies that, for all  $k\geq
(1+C)/\alpha$, 
$$
\sum_{j\geq0} (b_j\wedge b_{j+k}) \geq \sum_{j\geq j_0} b_{j+k} - C\sum_{j\geq
  j_0} (j+k+1)^{-\alpha-1}=\sum_{j\geq 0} b_{j+k}- \sum_{j=0}^{j_0-1} b_{j+k} - C\sum_{j\geq
  j_0} (j+k+1)^{-\alpha-1}  \;. 
$$
Since $j_0$ is fixed the two last term in the previous display are
$O(k^{-\alpha})$ as $k\to\infty$ and we conclude
from~(\ref{eq:ub-sum-wedge-lemma}). 
\end{proof}

We can now provide the proof of Proposition~\ref{prop:other-seeds}.

\begin{proof}[Proof of Proposition~\ref{prop:other-seeds}]
  From what precedes, we know that under these assumptions, the trawl process
  has an auto-covariance function of the form
  \begin{align}
    \label{eq:decomp-r-general-seed}
  r(k)=A\;S(k)+B\; P(k)\;,\qquad
    k\in\N\;,    \\
    \nonumber
    \text{with}\quad S(k)=\sum_{j=0}^\infty (\tilde a_j\wedge
\tilde a_{j+k})\quad    \text{and}\quad P(k)=\sum_{j=0}^\infty \tilde a_j\tilde a_{j+k}\;.
  \end{align}
  where $A>0$, $B\in\rset$ and $\tilde a_j= a_j\1_{\{a_j<\bar a\}}$ with
  $\bar a$ some positive constant.  We treat the two terms in the right-hand
  side of~(\ref{eq:decomp-r-general-seed}) separately.

\noindent\textbf{Term} $S$: Since $\tilde a_k=a_k$ for $k$ large enough, $(\tilde a_k)$ also satisfies
Condition~(\ref{eq:g4-bis}) and Lemma~\ref{lem:sum-wedge-lemma} gives
that
\begin{equation}
  \label{eq:Sk:first}
S(k)=\sum_{j\geq k} \tilde a_j + O(k^{-\alpha^*})=\sum_{j\geq k} a_j + O(k^{-\alpha^*})\;.
\end{equation}
Recall the definition of $r^{(d)}$
in~(\ref{eq:farim-cov}). Define, for all $k\geq0$,
$$
a_k^*=r^{(d^*)}(k)-r^{(d^*)}(k+1)=r^{(d^*)}(k)\frac{1-2d^*}{k+1-d^*}\;,
$$
where the second equality is derived in \cite[Example~5]{DJLS-spa}. By
\cite[Theorem~72.1]{gir2012} and its proof, we have for any $d\in(-1/2,1/2)$, 
\begin{equation}
  \label{eq:asymp-autocov-arfima}
  r_k^{(d)}=\frac{\Gamma(1-2d)}{\Gamma(1-d)\Gamma(d)}\;k^{-1+2d}\,(1+O(k^{-1}))\;.
\end{equation}
Hence the previous equation and the definition of $d^*$ give that
\begin{equation}
  \label{eq:asymp-autocov-diff-arfima}
a_k^*= \frac{\Gamma(2-2d^*)}{\Gamma(1-d^*)\Gamma(d^*)}\;k^{-\alpha^*}\,(1+O(k^{-1}))\;.
\end{equation}
And Condition~(\ref{eq:g4-bis}) is equivalent to have
$$
a_k=c^*\,\left(a^*_k+O(k^{-\alpha^*-1})\right)\,
$$
with $c^*>0$ only depending on $c$ and $\alpha^*$. Inserting this
in~(\ref{eq:Sk:first}) and using the definition of $a^*$, we obtain
$$
S(k)=c^* \, r^{(d^*)}(k) + O(k^{-\alpha^*})\;. 
$$
This, with the definition~(\ref{eq:farim-cov}) implies
$$
S(k)=\intperarg{\left(c^*\left|1-\rme^{-\rmi\lambda}\right|^{-2d^*}+h^*_S\right)\rme^{\rmi\lambda
    k}}{\lambda}\,,\qquad k\in\N\;,   
$$
where $h^*_S\in H(\alpha^*-1,C_S)$ for some $C_S>0$. 

\noindent\textbf{Term} $P$: It only remains to prove that $P$ defined
in~(\ref{eq:decomp-r-general-seed}) satisfies $P(k)=O(k^{-\alpha^*})$ as
$k\to\infty$ (so that the associated Fourier series 
belongs to  $H(\alpha^*-1,C_P)$ for some $C_P>0$). This follows immediately by
observing that~(\ref{eq:g4-bis}) with $\alpha^*>1$ implies, for some constant
$C>0$ and all $k\in\N$, 
$$
|P(k)|\leq C\,\sum_{j=0}^{\infty}(j+1)^{-\alpha^*}(j+k+1)^{-\alpha^*}\leq C\,\left(\sum_{j=0}^{\infty}(j+1)^{-\alpha^*}\right)\,(k+1)^{-\alpha^*}
\;.
$$
This concludes the proof. 
\end{proof}

\subsection{Convergence of the empirical covariance function}
\label{sec:postponed-proofs}
We start with the proof of Lemma~\ref{lem:cons-emp-cov-centered}.
\begin{proof}[Proof of Lemma~\ref{lem:cons-emp-cov-centered}]
    Let $r$ denote the autocovariance function of $X$. 
  Let $t_1\leq t_2 \leq t_3\leq t_4$ in $\Z$. We use the identities
  \begin{align}
    \nonumber
\cov(X_{t_1}X_{t_2}\ ,\ X_{t_3}X_{t_4})=  \cum{X_{t_1}\ ,\ X_{t_2}\ ,\ X_{t_3}\ ,\ X_{t_4}}&+
                                                                                             r(t_1-t_3)r(t_2-t_4)\\
    \label{eq:identity-cum4covpairwise}
                                                                                           &+r(t_1-t_4)r(t_2-t_3)\;,\\
\cov(X_{t_1}X_{t_2}X_{t_3}\ ,\ X_{t_4})=  \cum{X_{t_1}\ ,\ X_{t_2}\ ,\ X_{t_3}\ ,\ X_{t_4}}&
    \nonumber
       +r(t_4-t_1)r(t_3-t_2)\\
    \nonumber
                                                                                           &+r(t_4-t_2)r(t_3-t_1)\\
    \label{eq:identity-cum4cov13}
    &+r(t_4-t_3)r(t_2-t_1)\;.     
  \end{align}
  This, with the
  bound~(\ref{eq:cov-gen-ass-bound}), allows to go back and forth
  from~(\ref{eq:pairwise-cov-gen-ass-bound})
  or~(\ref{eq:13wise-cov-gen-ass-bound}) to~(\ref{eq:cum-gen-ass-bound}).
\end{proof}
We can now prove Proposition~\ref{prop:cons-emp-cov-centered}.
\begin{proof}[Proof of Proposition~\ref{prop:cons-emp-cov-centered}]
  We have, using again the identity displayed in~(\ref{eq:cum-gen-ass-bound}),
  \begin{align}
    \nonumber
  \abs{\cov(\widetilde r_n(k)\ ,\ \widetilde r_n(\ell))}&\leq \frac1{n^2}\sum_{s=1}^{n-k}\sum_{s'=1}^{n-\ell}
                                                          \abs{\cov(X_{s}X_{s+k}\ ,\ X_{s'}X_{s'+\ell})}\\
        \label{eq:basic-bound-centered-emp-cov-I}
                                                        &\leq\frac1{n^2}\sum_{s=1}^{n-k}\sum_{s'=1}^{n-\ell}
                                                          \abs{\cum{X_{s}\ ,\
                                                          X_{s+k}\ ,\ X_{s'}\
                                                          ,\ X_{s'+\ell}}}\\
        \label{eq:basic-bound-centered-emp-cov-II}
                                                        &+\frac1{n^2}\sum_{s=1}^{n-k}\sum_{s'=1}^{n-\ell}\abs{r(s-s')r(s-s'+k-\ell)}\\
        \label{eq:basic-bound-centered-emp-cov-III}
        &+\frac1{n^2}\sum_{s=1}^{n-k}\sum_{s'=1}^{n-\ell}\abs{r(s-s'+k)r(s-s'-\ell)}
                                                          \;.    
  \end{align}
  Using~(\ref{eq:cov-gen-ass-bound}), we get
  that~(\ref{eq:basic-bound-centered-emp-cov-II})
  and~(\ref{eq:basic-bound-centered-emp-cov-III}) are both less than or equal to
  $$
  C_0^2\frac1{n^2}\sum_{s=1}^{n}\sum_{s'=1}^{n}(1+|s-s'|)^{-s_0}\leq
  C_0^2\frac1{n}\sum_{\tau=-n+1}^{n-1}(1+|\tau|)^{-s_0}\leq C\ n^{-s_0}\;,
  $$
  where $C>0$ only depends on $C_0$ and $s_0$. To
  get~(\ref{eq:empcov-cov-gen-ass}), it thus only remains to show that a
  similar bound holds for the term appearing
  in~(\ref{eq:basic-bound-centered-emp-cov-I}). To this end we use the
  bound~(\ref{eq:cum-gen-ass-bound}) that we have showed to hold
  under~\ref{item:hyp-consistency-common} in
  Lemma~\ref{lem:cons-emp-cov-centered}. More precisely we use the bound on left-hand
  side of the $\wedge$ sign in ~(\ref{eq:cum-gen-ass-bound}) in the first
  following case and  the bound on right-hand
  side of the $\wedge$ sign for all the other cases:
  \begin{enumerate}
  \item For $s'\leq s\leq s+k\leq s'+\ell$,
    $$
    \abs{\cum{X_{s}\ ,\ X_{s+k}\ ,\ X_{s'}\ ,\ X_{s'+\ell}}}\leq C_2\ (1+|s-s'|)^{-s_0}\;.
    $$
  \item For $s\leq s', s+k\leq s'+\ell$,
    $$
    \abs{\cum{X_{s}\ ,\ X_{s+k}\ ,\ X_{s'}\ ,\ X_{s'+\ell}}}\leq C_2\ (1+|s+k-s'|)^{-s_0}\;.
    $$
  \item For $s'\leq s, s'+\ell\leq \leq  s+k$,
    $$
    \abs{\cum{X_{s}\ ,\ X_{s+k}\ ,\ X_{s'}\ ,\ X_{s'+\ell}}}\leq C_2\ (1+|s'+\ell-s|)^{-s_0}\;.
    $$
  \end{enumerate}
(The case $s\leq s'\leq s'+\ell\leq s+k$ can only occur if $s=s'$ and $\ell=k$
since we assumed $0\leq k\leq \ell$, so is included in the first case.)
 Hence we get that the term in~(\ref{eq:basic-bound-centered-emp-cov-I}) is
 bounded from above by
 $$
 \frac{C_2}{n^2}\ \max_{j=0,k,\ell}\ \sum_{s=1}^{n}\sum_{s'=1}^{n}
 (1+|s-s'+j|)^{-s_0}\leq \frac{C_2}{n}\
 \sum_{\tau=-2n}^{2n}(1+|\tau|)^{-s_0}\leq C\ n^{-s_0}\;,
 $$
 where $C$ only depends on $C_2$ and $s_0$. 
\end{proof}
Next, we prove Corollary~\ref{cor:covar-estim}.
\begin{proof}[Proof of Corollary~\ref{cor:covar-estim}]
  Since $\hat r_n(k)$, $I_n$, $r$ and $f$ are invariant by centering, we can assume
  in the following that $X$ is centered without loss of generality. 

  Using Proposition~\ref{prop:cons-emp-cov-centered} and  $\E\, \tilde
  r_n(k)=(1-k/n)\ r(k)$, we have, for all $0\leq k<n$, 
  $$
  \E\,  \left(\tilde r_n(k)-r(k)\right)^2\leq C\,n^{-s_0}+\left(\frac kn \
    r(k)\right)^2
  \leq n^{-s_0}\left(C\,+C_0(1+k)^{2-2s_0}n^{s_0-2}\right) \leq n^{-s_0}\,(C\,+C_0)\;,
    $$
    where the second inequality follows from~(\ref{eq:cov-gen-ass-bound})
    in~\ref{item:hyp-consistency-common}. The same bound gives that, for all $\ell\geq1$,
  \begin{equation}
    \label{eq:var-sum-trivial} 
  \v\ \sum_{j=1}^{\ell}X_j=\sum_{\tau=-\ell+1}^{\ell-1}(\ell-|\tau|)r(\tau)\leq
  \frac{2C_0}{1-s_0} \, (2+\ell)^{2-s_0}\;,    
  \end{equation}
 and  we get, with~(\ref{eq:decomp-emp-cov-remainder}), for all $0\leq k<n$,
  $$
  \E \abs{R^r_n(k)}\leq n^{-2}\left(\v\ \sum_{j=1}^{n}X_j
    +\left(\v\
      \sum_{j=1}^{n}X_j\right)^{1/2}\left(\v\ \sum_{j=k+1}^{n-k}X_j\right)^{1/2}
  \right)\leq  \frac{2C_0}{1-s_0} \, (2+n)^{-s_0}\;.
$$
With~(\ref{eq:decomp-emp-cov}), we conclude
that~(\ref{eq:empcov-cov-noncentered}) holds.
\end{proof}

Let $\tore$ denote $\R/2\pi\Z$. Recall that the spectral measure $\mu$ of  a
weakly stationary process  $X$ is a finite measure on $\tore$ such that the
covariance function of $X$ 
satisfies
$$
r(k)=\int \rme^{\rmi\lambda}\;\mu(\rmd\lambda)\;,\quad k\in\Z\;.
$$
We derive the following useful lemma.
\begin{lem}\label{lem:parseval-like-decomp}
    Let $X$ be weakly stationary process with spectral measure $\mu$ and define
    the periodogram and the empirical covariance by~(\ref{eq:periodo-def})
    and~(\ref{eq:emp-cov-def}). Let $h:\R\to\R$
be a  $(2\pi)$-periodic bounded function. Then,
   for all $0\leq m<n$ and $(c_k)_{-m\leq k\leq
     m}\in\C^{2m+1}$,
   $$
 \abs{\int_{\tore}{I_n\, h}-\int h\;\rmd\mu}\leq\sum_{k=-m}^m\abs{\hat
   r_n(k)-r(k)}\,\abs{c_k}+\sup_{\lambda\in\R}\abs{h(\lambda)-\sum_{k=-m}^mc_k\rme^{\rmi\lambda
   k}}\,\left(\hat r_n(0)+r(0)\right)\;.
$$
\end{lem}
\begin{proof}
Let  $0\leq m<n$ and $(c_k)_{-m\leq k\leq
    m}\in\C^{2m+1}$, and denote $h_m(\lambda)=\sum_{k=-m}^mc_k\rme^{\rmi\lambda
    k}$ and $\epsilon_m=\sup|h-h_m|$. We write
  $$
  \int_{\tore}{I_n\, h}-\int h\;\rmd\mu=
  \int_{\tore}{I_n\, h_m}-\int h_m\;\rmd\mu+
  \int_{\tore}{I_n\, (h-h_m)}-\int (h-h_m)\;\rmd\mu\;.
  $$
  Replacing $h_m$ by its definition, we get
  $$
    \int_{\tore}{I_n\, h_m}-\int h_m\;\rmd\mu=\sum_{k=-m}^mc_k\,(\hat r_n(k)-r(k))\;.
  $$
  Then, by definition of $\epsilon_m$, we have
  $$
\abs{  \int_{\tore}I_n\,(h-h_m)}\leq\epsilon_m\,\int_{\tore}{I_n}=\epsilon_m\,\hat r_n(0)\;,
$$
and, similarly $\abs{  \int{(h-h_m)}}\rmd\mu\leq\epsilon_m\,r(0)$. The
result follows.
\end{proof}
We can no prove Corollary~\ref{cor:spdens-estim}.
\begin{proof}[Proof of  Corollary~\ref{cor:spdens-estim}]
  For $\bandwidth$ small enough, since $J$ is compactly supported, we have, for all
  $\lambda_0\in\R$ and $k\in\Z$,
  $$
  c_k(J_{\bandwidth,\lambda_0})=\frac1{2\pi \bandwidth}\int
  J((\lambda-\lambda_0)/\bandwidth)\,\rme^{\rmi\lambda k}\;\rmd\lambda=\rme^{-\rmi\lambda_0 k}\,J^*(\bandwidth\,k)\;,
  $$
  where $J^*(\xi)=\int J(x)\,\rme^{\rmi x\xi}\,\rmd x$ is the Fourier transform
  of $J$.   Since $J$ is two times continuously differentiable and has compact
  support, we have $\abs{J^*(\xi)}={\cal O}(|\xi|^{-2})$ as
  $|\xi|\to\infty$. Hence $J_{\bandwidth,\lambda_0}$ has absloutely summable Fourier
  coefficients and the following identity holds
  $$
  J_{\bandwidth,\lambda_0}(\lambda)=\sum_{k\in\Z}  c_k(J_{\bandwidth,\lambda_0})\,\rme^{\rmi\lambda k}\;,\quad\lambda\in\R\;.
  $$
  Applying Lemma~\ref{lem:parseval-like-decomp} with $m=n-1$ and $c_k=
  c_k(J_{\bandwidth,\lambda_0})$ we get that
   $$
   \abs{\int_{\tore}{I_n\, J_{\bandwidth,\lambda_0}}-\int J_{\bandwidth,\lambda_0}\;\rmd\mu}\leq\sum_{k=-n}^n\abs{\hat
   r_n(k)-r(k)}\,\abs{J^*(\bandwidth\,k)}+\left(\sum_{|k|>n}\abs{J^*(\bandwidth\,k)}\right)\,\left(\hat r_n(0)+r(0)\right)\;.
$$
Applying~(\ref{eq:empcov-cov-noncentered}) and  $\abs{J^*(\xi)}={\cal O}(|\xi|^{-2})$, we get 
$$
 \abs{\int_{\tore}{I_n\, J_{\bandwidth,\lambda_0}}-\int J_{\bandwidth,\lambda_0}\;\rmd\mu}\leq
 C'\,n^{-s_0/2}\,\left(\sum_{k=-n}^n\abs{J^*(\bandwidth\,k)}\right)+ C_1\,(\bandwidth\,n)^{-1}\,r(0)(1+C')\;,
$$
where $C_1$ only depends on $J$. The result then follows from the fact that
$$
\lim_{n\to\infty,\bandwidth\to0}\frac1\bandwidth\sum_{k=-n}^n\abs{J^*(\bandwidth\,k)}=\int |J^*|
<\infty\;. 
$$
\end{proof}
\subsection{Consistency of parametric Whittle estimation}
\label{sec:cons-param-whittle}

We first introduce some notation valid throughout this section and derive
useful lemmas. For any $d\in\R$ and $\epsilon>0$, we define
\begin{align*}
  \underline{f}_{(d,\epsilon)}(\lambda)=\frac{2^{-2(d+\epsilon)}}{2\pi}\, \left|\sin\frac\lambda2\right|^{-2(d-\epsilon)}\quad\text{and}\quad
  \overline{f}_{(d,\epsilon)}(\lambda)= \frac{2^{-2(d-\epsilon)}}{2\pi}\,\left|\sin\frac\lambda2\right|^{-2(d+\epsilon)}\;,
\end{align*}
so that, for all $d'\in[d-\epsilon,d+\epsilon]$ and $\lambda\in\R$,
\begin{equation}
    \label{eq:functional-bounds}
\underline{f}_{(d,\epsilon)}(\lambda)\leq f_{d'}(\lambda)=
\frac{2^{-2d'}}{2\pi}\,\left|\sin\frac\lambda2\right|^{-2d'}\leq  \overline{f}_{(d,\epsilon)}(\lambda)\;.
\end{equation}
Finally we denote
\begin{equation}
    \label{eq:pos-param-dens}
    a_K:= \inf_{(d,h)\in K}\inf_{\lambda\in\R}\left(f_d(\lambda)+h(\lambda)\right)\;.
  \end{equation}
 We now introduce the useful lemmas.  
\begin{lem}
  \label{lem:lipschitz-contrast}
  Let $a>0$ and $f:\R\to\R$ be $(2\pi)$-periodic. Let $g:\R\to\R_+$ be
  $(2\pi)$-periodic and such that $\intper{g}>0$. Then
  $h\mapsto\ln\intper{\frac{g}{(f+h)\vee a}}$ is $(1/a)$-Lipschitz on
  $\cont$. If moreover $\intper{\ln (f\vee a) }<\infty$, $h\mapsto\intper{\ln
    ((f+h)\vee a) }$  is also $(1/a)$-Lipschitz on
  $\cont$.
\end{lem}
\begin{proof}
We apply successively that, for all $0<x\leq y$, 
  $$
      \left|\frac1x-\frac1y\right|=\frac{|y-x|}{xy}
\quad\text{and}\quad
  \left|\ln y-\ln x\right|\leq\frac{|y-x|}{x}\;.
    $$
    We obtain, for all $h,\tilde h\in\cont$,
  \begin{align*}
\left|\intper{\frac{g}{(f_d+h)\vee a}}-\intper{\frac{g}{(f_d+\tilde h)\vee
      a}}\right|&\leq
                  \intper{\frac{g\,|h-\tilde h|}{((f_d+h)\vee a)\,((f_d+\tilde h)\vee a)}}    \\
    &\leq\frac1a \, \sup|h-\tilde h| \;\intper{\frac{g}{((f_d+h)\vee
      a)}}\\
        &\text{ or }\leq\frac1a \, \sup|h-\tilde h|
          \;\intper{\frac{g}{((f_d+\tilde h)\vee
      a)}}\;,
  \end{align*}
  hence is bounded from above by the min of the two last right-hand sides. 
  Taking the difference of the log's then yields 
  $$
  \left|\ln \intper{\frac{g}{(f_d+h)\vee a}}-\ln \intper{\frac{g}{(f_d+\tilde h)\vee
      a}}\right|\leq \frac1a \, \sup|h-\tilde h|\;. 
$$
Hence we get the first assertion.

Similarly, we get that, for all $h,\tilde h\in\cont$,
  $$
  \left|\ln ((f_d+h)\vee a)-\ln ((f_d+\tilde h)\vee
        a)\right|\leq\frac1a \,  \sup|h-\tilde h|\;. 
      $$
And we get the second assertion.
\end{proof}
\begin{lem}
  \label{lem:some-unif-ounds-for-fd}
  Let $d^*<1/2$, $h^*\in\cont$ and $a>0$. 
  For all $d\in\R$ and $h\in\cont$, we have
  \begin{align*}
\lim_{\epsilon\to0} \ln\intper{\frac{f_{d^*}+h^*}{({f}_{(d,\epsilon)}+h)\vee
      a}} =
    \ln\intper{\frac{f_{d^*}+h^*}{(f_d+h)\vee a}}  \;,  \\
\lim_{\epsilon\to0}     \intper{\ln(({f}_{(d,\epsilon)}+h)\vee
      a)}= \intper{\ln((f_d+h)\vee a)}\;,
  \end{align*}
  where ${f}_{(d,\epsilon)}$ denotes either $\underline{f}_{(d,\epsilon)}$ or $\overline{f}_{(d,\epsilon)}$.
\end{lem}
\begin{proof}
  We have, for all $\lambda\notin2\pi\Z$, as $\epsilon\to0$,
  ${f}_{(d,\epsilon)}(\lambda)\to f_d(\lambda)$. Moreover, for $\epsilon\in(0,1)$,
  $$
  \frac{f_{d^*}+h^*}{({f}_{(d,\epsilon)}+h)\vee
    a}\leq \frac1a \, (f_{d^*}+h^*)\quad\text{and}\quad
\ln a\leq \ln(({f}_{(d,\epsilon)}+h)\vee a) \leq
\ln((\overline{f}_{(d,1)}+h)\vee a)\;.
$$
  We conclude by dominated
  convergence.  
\end{proof}
\begin{lem}
  \label{lem:some-continuity-using-fd}
  Let $d^*<1/2$, $h^*\in\cont$ and $a>0$. 
  For all $d\in\R$ and $h\in\cont$, we have
  \begin{align*}
\lim_{(d',\tilde h)\to(d,h)} \ln\intper{\frac{f_{d^*}+h^*}{(f_{d'}+\tilde h)\vee
      a}} = 
    \ln\intper{\frac{f_{d^*}+h^*}{(f_d+h)\vee a}}  \;,  \\
\lim_{(d',\tilde h)\to(d,h)}     \intper{\ln((f_{d'}+\tilde h)\vee
      a)}= \intper{\ln((f_d+h)\vee a)}\;.
  \end{align*}
\end{lem}
\begin{proof}
  By Lemma~\ref{lem:lipschitz-contrast}, we have for all $d'\in\R$ and
  $h,\tilde h\in\cont$, 
  $$
\left|  \ln\intper{\frac{f_{d^*}+h^*}{(f_{d'}+\tilde h)\vee
      a}} - \ln\intper{\frac{f_{d^*}+h^*}{(f_{d'}+h)\vee
      a}}\right| \leq\frac1a\,\sup|h-\tilde h|\;.
  $$
  Then using~(\ref{eq:functional-bounds}) and
  Lemma~\ref{lem:some-unif-ounds-for-fd}, we get the first assertion.
  The second assertion is proved similarly. 
\end{proof}
\begin{lem}\label{lem:pos-param-dens}
  Let $K$ be a compact subset of $\R\times\cont$ such that, for all
  $(d,h)\in K$, $f_d+h>0$ on $\R$. Suppose moreover that
    \begin{equation}
    \label{eq:strange-condition}
\left(  \{0\}\times\cont\right) \cap
  \closure_K(\R_-^*\times\cont)=\emptyset\;,    
\end{equation}
where for any $A\subset\rset\times\cont$, $\closure_K(A)$ denotes the closure of $A\cap K$ in $K$. 
  Then we have $\inf_{(d,h)\in K}\inf(f_d+h)>0$. 
\end{lem}
\begin{proof}
  Let
  $$
  C_K=\sup_{(d,h)\in K,\lambda\in\R}|h(\lambda)|\;,
  $$
  which is finite since $K$ is compact and $(h,\lambda)\mapsto h(\lambda)$ continuous.
  
  We will need the following remark.   Let $\epsilon\in(0,\pi)$. We have, for all
  $d>0$ and  $\lambda\in[-\epsilon,\epsilon]$, 
  $$
  f_d(\lambda)=\frac{2^{-2d}}{2\pi}\,\left|\sin(\lambda/2)\right|^{-2d}\geq\frac{\epsilon^{-2d}}{2\pi}\;.
  $$
  Then we get, if $\epsilon^{-2d}/(2\pi)\geq 2C_K+1$, which is equivalent to
  $d\geq\ln(2\pi(2C_K+1))/(-2\ln\epsilon)$, for all $(d,h)\in K$ and
  $\lambda\in[-\epsilon,\epsilon]$,
  \begin{equation}
    \label{eq:bound-CK-small-lamb}
    f_d(\lambda)+h(\lambda)\geq \epsilon^{-2d}/(2\pi)-C_K\geq C_K+1\;.
  \end{equation}
  Let $(d_n,h_n,\lambda_n)$ be a sequence valued in $K\times[-\pi,\pi]$ such
  that
  $$
  a_K=\lim_{n\to\infty}\left(f_{d_n}(\lambda_n)+h_n(\lambda_n)\right)\;,
  $$
  where $a_K$ is defined by~(\ref{eq:pos-param-dens}).  By compactness, there
  is an increasing sequence of integers $(q_n)$ and
  $(d,h,\lambda)\in K\times[-\pi,\pi]$ such that
  $(d_{q_n},h_{q_n},\lambda_{q_n})$ converges to $(d,h,\lambda)$. We now
  separate four cases, which cover all possible cases: 1) $\lambda\neq0$, 2)
  $d<0$, 3) $\lambda=0$ and $d>0$ 4) $\lambda=d=0$.
  
  \noindent\textbf{Case 1)} Suppose that $\lambda\neq0$. Since the mapping
  $(d',\tilde h,\lambda')\mapsto f_{d'}(\lambda')+\tilde h(\lambda')$ is continuous on
  $\R\times\cont\times(\R\setminus(2\pi\Z))$, we get that
  $a_K=f_d(\lambda)+h(\lambda)>0$.

  \noindent\textbf{Case 2)} Suppose that $d<0$. This case is similar to Case
  1): it is sufficient to show that
  $(d',\tilde h,\lambda')\mapsto f_{d'}(\lambda')+\tilde h(\lambda')$ is
  continuous on $(-\infty,0)\times\cont\times\R$, which follows from the
  continuity of $(x,u)\mapsto u^x$ on $(0,\infty)\times\R_+$, which is easy to
  establish.

  \noindent\textbf{Case 3)} Suppose that $\lambda=0$ and $d>0$. Then there
  exists an arbitrarily small $\epsilon\in(0,\pi)$ such that $d\geq
  \ln(2\pi(2C_K+1))/(-4\ln\epsilon)$, and
  ~(\ref{eq:bound-CK-small-lamb}) implies, for $n$ large enough,
  $f_{d_{q_n}}(\lambda_{q_n})+h_{q_n}(\lambda_{q_n}) \geq C_K+1$ hence $a_K>0$. 
  
  \noindent\textbf{Case 4)} Suppose that $\lambda=d=0$. Thanks to Condition~(\ref{eq:strange-condition}), $d_{q_n}$ must be non-negative for $n$ large enough,
  in which case we have
  $$
  f_{d_{q_n}}(\lambda_{q_n})=\frac{2^{-2d_{q_n}}}
  {2\pi}\,\left|\sin(\lambda_{q_n}/2)\right|^{-2d_{q_n}}\geq\frac{2^{-2d_{q_n}}}
  {2\pi}\;,
  $$
  which tends to $1/(2\pi)$ as $n\to\infty$, while $h_{q_n}(\lambda_{q_n})$ tends to
  $h(0)$. Hence
  $$
  a_K= \lim_{n\to\infty}f_{d_{q_n}}(\lambda_{q_n})+h_{q_n}(\lambda_{q_n})\geq
  1/(2\pi)+h(0)=f_d(\lambda)+h(\lambda)\;,
  $$
  since $d=\lambda=0$. Again, we get $a_K>0$. 
\end{proof}
\begin{Rem}
  Condition~(\ref{eq:strange-condition}) means that any parameter $(0,h)$ in
  $K$ is isolated from parameters $(d,\tilde h)$ with $d<0$. This assumption
  cannot be avoided in Lemma~\ref{lem:pos-param-dens}.(As a counterexample,
  take $K=\{(d,-d)~:~d\in[-1/2,0]\}$ where here $-d$ is seen as the function in
  $\cont$ that is constant equal to $-d$). It is of course trivially satisfied
  if $K\subset\R_+\times\cont$.
\end{Rem}
We can now proceed with the proofs of Lemma~\ref{lem:Kn-infinite-dim} and
Theorem~\ref{theo:cons}.
\begin{proof}[Proof of Lemma~\ref{lem:Kn-infinite-dim}]
  We first recall why $H(s,C)$ is a compact subset of $\cont$. 
  For all $u,v\in\R$, we have
  $$
  |\rme^{\rmi k u}-\rme^{\rmi k v}|=\left|\int_0^{|u-v|}\rmi k\,\rme^{\rmi k
      x}\;\rmd x\right|\leq k\,|u-v|\;.
  $$
  We get that, for all $h\in H(s,C)$ and $u,v\in\R$,
  $$
  \left|h(u)-h(v)\right|\leq C\,\left(|u-v|\sum_{|k|\leq
      |u-v|}(1+|k|)^{-1-s}\,k+2\sum_{|k|> |u-v|}(1+|k|)^{-1-s}\right)={\cal O}(|u-v|^{s\wedge1})\;,
  $$
  where the ${\cal O}$ does not depend on $h$.  By the Arzelà–Ascoli theorem,
  we get that $H(s,C)$ is a compact subset of $\cont$. It follows that $A$, as
  a closed subset of $[0,1/2]\times\cont$ is also compact. Thus by
  Lemma~\ref{lem:pos-param-dens}, there exists $a_K>0$ such that
  $f_d+h\geq a_K$ for all $(d,h)\in A$. Since $\sup(|p_m(h)-h|)$ tends to 0
  uniformly in $h\in H(s,C)$ as $m\to\infty$, we get that there exists a
  positive integer $m_0$ such that $f_d+p_m[h]\geq a_K/2>0$ on $\rset$ for all
  $(d,h)\in A$ and $m\geq m_0$. Let $K_n$ and $K$ be defined as in the lemma
  for some diverging sequence $(m_n)$ of integers larger than or equal to
  $m_0$. It is straightforward to show that $K$ is compact (because for any
  increasing or constant sequence $(\alpha_k)_{k\in\nset}$ of integers and any
  sequence $((d_k,h_k))_{k\in\nset}$ valued and converging in $A$, we have that
  $(d_k,p_{m_{\alpha_k}}[h_k])$ converges in $K$).
  Assumption~\ref{item:hyp-consistency} easily follows by setting
  $h^*_n=p_{m_n}[h^*]$.
\end{proof}
\begin{proof}[Proof of Theorem~\ref{theo:cons}]
Define
$$
\Lambda^*:= \ln c^*+\intper{\ln(f_{d^*}+h^*)} \;.
$$
By~\ref{item:hyp-consistency}, we have $(d^*,h^*_n)\in K_n$. Thus
Equation~(\ref{eq:def-param-edstim}) implies $(\hat d_n,\hat h_n)\to(d^*,h^*)$
in $[0,1/2]\times\cont$ a.s. provided that
\begin{equation}
  \label{eq:proof-cons-theta-star}
\limsup_{n\to\infty}\Lambda_n(d^*,h^*_n)\leq \Lambda^*\quad\text{a.s.}  
\end{equation}
and that, for any $\epsilon_0>0$,
we have
\begin{equation}
  \label{eq:proof-cons-away-theta-star}
\liminf_{n\to\infty}\inf_{(d,h)\in K_0}\Lambda_n(d,h)> \Lambda^*\quad\text{a.s.}  \;,
\end{equation}
where $K_0$ is defined by
$$
K_0=\left\{(d,h)\in K~:~|d-d^*|+\sup|h-h^*|\geq\epsilon_0\right\}\;,
$$
We start with the proof of~(\ref{eq:proof-cons-theta-star}).
By definition of $a_K$ in~(\ref{eq:pos-param-dens}), we have, for
all $(d,h)\in K$,
\begin{align}\label{eq:Lambdastar-aK}
&  \Lambda^*:= \ln c^*+\intper{\ln((f_{d^*}+h^*)\vee a_K)} \;,\\
\label{eq:LambdaNstar-aK}&\Lambda_n(d,h)=  \ln\intper{\frac{I_n}
  {(f_d+h)\vee a_K}}\;+\intper{\ln \left((f_d+h)\vee a_K\right)}
\; .
\end{align}
And by Lemma~\ref{lem:pos-param-dens}, $a_K>0$. Note that
$\intper{I_n}=\hat r_n(0)/(2\pi)>0$ for $n$ large enough, a.s. Applying
Lemma~\ref{lem:lipschitz-contrast} with $a=a_K$ and since $h^*_n$ converges to
$h^*$ uniformly by~\ref{item:hyp-consistency}, we get
$$
\lim_{n\to\infty}\left|\Lambda_n(d^*,h^*_n)-\Lambda_n(d^*,h^*)\right|=0\quad\text{a.s.}
$$
Since $1/((f_{d^*}+h^*)\vee a_K)$ is continuous and $X$ is ergodic with
spectral density $f$ given by~(\ref{eq:sp-dens-general}), we have
$$
\lim_{n\to\infty}\intper{\frac{I_n}
  {(f_{d^*}+h^*)\vee a_K}}=\intper{\frac{c^*(f_{d^*}+h^*)}
  {(f_{d^*}+h^*)\vee a_K}}\quad\text{a.s.}
$$
(see e.g. \cite[Theorem~8.2.1]{gir2012}). By definition of $a_K$, this limit is
$c^*$ and, with the three previous displayed equation, we get~(\ref{eq:proof-cons-theta-star}).

We conclude with the proof of~(\ref{eq:proof-cons-away-theta-star}), given some
$\epsilon_0>0$. Equations~(\ref{eq:LambdaNstar-aK})
and~(\ref{eq:functional-bounds}) and Lemma~\ref{lem:lipschitz-contrast} with
$a=a_K$ yield,
for all $(d,h)$ and $(d',\tilde h)$ in $K$ such that $|d-d'|\leq\epsilon$ and
$\sup|h-\tilde h|\leq\epsilon$,
\begin{equation*}
  \Lambda_n(d',\tilde h)\geq
 \ln\intper{\frac{I_n}
  {(\overline{f}_{(d,\epsilon)}+h)\vee a_K}}\;+\intper{\ln \left((\underline{f}_{(d,\epsilon)}+h)\vee a_K\right)}-\frac{2\epsilon}{a_K}
\; .
\end{equation*}
Since $1/((\overline{f}_{(d,\epsilon)}+h)\vee a_K)$ is continuous and $X$ is ergodic, we have
$$
\lim_{n\to\infty}\intper{\frac{I_n}
  {(\overline{f}_{(d,\epsilon)}+h)\vee a_K}}=\intper{\frac{c^*(f_{d^*}+h^*)}
  {(\overline{f}_{(d,\epsilon)}+h)\vee a_K}}\quad\text{a.s.}
$$
The last two displays give that, for all  $(d,h)\in K$ and $\epsilon>0$,
\begin{multline}
  \label{eq:LB-LambdaN}
\liminf_{n\to\infty}\inf\left\{\Lambda_n(d',\tilde h)~:~(d',\tilde h)\in K,\,|d'-d|\leq\epsilon,\,\sup|h-\tilde h|\leq\epsilon\right\}\\\geq
 \ln\intper{\frac{c^*(f_{d^*}+h^*)}
  {(\overline{f}_{(d,\epsilon)}+h)\vee a_K}}\;+\intper{\ln \left((\underline{f}_{(d,\epsilon)}+h)\vee a_K\right)}-\frac{2\epsilon}{a_K}\quad\text{a.s.}
\end{multline}
For all $(d,h)\in K_0$, since $f_d+h$ and $f_{d^*}+h^*$ do not coincide almost
everywhere, the Jensen inequality and the definition of $\Lambda^*$ give that
$$
\ln\intper{\frac{c^*(f_{d^*}+h^*)}{(f_d+h)\vee a_K}}+\intper{\ln((f_d+h)\vee a_K)}>\Lambda^*\;.
$$
Since $K_0$ is compact, by Lemma~\ref{lem:some-continuity-using-fd}, we can find $\eta>0$ such that
$$
\inf_{(d,h)\in K_0}\ln\intper{\frac{c^*(f_{d^*}+h^*)}{(f_d+h)\vee
    a_K}}+\intper{\ln((f_d+h)\vee a_K)}\geq\Lambda^*+4\eta\;.
$$
Applying Lemma~\ref{lem:some-unif-ounds-for-fd} with $a=a_K$, we get that, for
all $(d,h)\in K_0$, there exists $\epsilon>0$ such that
\begin{align*}
\ln\intper{\frac{c^*(f_{d^*}+h^*)}{(\overline{f}_{(d,\epsilon)}+h)\vee
    a_K}}+\intper{\ln((\underline{f}_{(d,\epsilon)}+h)\vee
  a_K)}
  &\geq
    \ln\intper{\frac{c^*(f_{d^*}+h^*)}{(f_d+h)\vee a_K}}\\&+\intper{\ln((f_d+h)\vee a_K)}-2\eta\\
   & \geq\Lambda^*+2\eta\;.  
\end{align*}
Since $K_0$ is compact, we can thus cover $K_0$ with a finite collection
$(B_i)_{i=1,\dots,N}$, for which, for any $i=1,\dots,N$, there exists
$\epsilon_i\in(0,\eta\,a_K/2)$ and $(d_i,h_i)\in
K_0$ such that
$$
\ln\intper{\frac{c^*(f_{d^*}+h^*)}{(\overline{f}_{(d_i,\epsilon_i)}+h_i)\vee
    a_K}}+\intper{\ln((\underline{f}_{(d_i,\epsilon_i)}+h_i)\vee
  a_K)}\geq\Lambda^*+2\eta\;,
$$
and  $(d,h)\in B_i$ implies $|d-d_i|\leq\epsilon_i$ and
$\sup|h-h_i|\leq\epsilon_i$. Let $i=1,\dots,N$. Applying~(\ref{eq:LB-LambdaN})
with $d=d_i$, $h=h_i$ and $\epsilon=\epsilon_i$, we
get that
$$
\liminf_{n\to\infty}\inf_{(d',\tilde h)\in B_i}\Lambda_n(d',\tilde h)\geq \Lambda^*+\eta\quad\text{a.s.}
$$
Since 
$(B_i)_{i=1,\dots,N}$ covers  $K_0$, we
obtain~(\ref{eq:proof-cons-away-theta-star}).

We thus have proved that $(\hat d_n,\hat h_n)\to(d^*,h^*)$
in $[0,1/2]\times\cont$ a.s. and it only remains to show that $\hat c_n\to c^*$
a.s., where, by~(\ref{eq:c-parm-estim}) and the definition of $a_K$,
$$
\hat c_{n}=\intper{\frac{I_n} {(f_{\hat d_{n}}+\hat h_{n})\vee a_K}}\;.
$$
Let $\epsilon>0$ and suppose that $|\hat d_n-d^*|\leq\epsilon$ and $\sup|\hat
h_n-h^*|\leq\epsilon$, which happens for $n$ large enough, a.s. 
By (\ref{eq:functional-bounds}) and Lemma~\ref{lem:lipschitz-contrast}
successively, we get that 
$$
\ln \intper{\frac{I_n} {(\overline{f}_{(d^*,\epsilon)}+h^*)\vee a_K}}-\frac\epsilon{a_K}\leq \ln \hat c_{n}
\leq \ln \intper{\frac{I_n} {(\underline{f}_{(d^*,\epsilon)}+h^*)\vee a_K}}+\frac\epsilon{a_K}
$$
Letting $n$ tend to $\infty$ and then $\epsilon$ to zero (using Lemma~\ref{lem:some-unif-ounds-for-fd}), we get that
$$
\lim_{n\to\infty}\ln \hat c_n=\ln\intper{\frac{c^*(f_{d^*}+h^*)}{(f_{d^*}+h^*)\vee a_K}}\quad\text{a.s.}
$$
By definition of $a_K$, the latter integral is $c^*$ and the proof is concluded.
\end{proof}

\subsection{The arfima case: Proof of Corollary~\ref{cor:arfima}}
\label{sec:arfima-case:-proof}
  Denote, for any $d\in[0,1/2)$ and $\vartheta=(\phi,\theta)\in\Theta_{p,q}$,
  $$
  f_{d,\vartheta,\sigma^2}(\lambda)=\sigma^2\,f_{d}(\lambda)\,\left|\frac{\Theta(\rme^{-\rmi k\lambda})}{\Phi(\rme^{-\rmi k\lambda})}\right|^2\;.  
  $$
  From the discussion preceding Corollary~\ref{cor:arfima} and leading
  to~(\ref{eq:hdtheta-arfima}), we write
  $$
  f_{d,\vartheta,\sigma^2}=\sigma^2\frac{\left|\Theta(1)\right|^2}{\left|\Phi(1)\right|^2}\left(f_d+f_d\mathrm{R}(\vartheta)\right)\;,
  $$
  with $\mathrm{R}(\vartheta)$ defined for all $\vartheta\in\Theta_{p,q}$ as the
  $\rset\to\rset$ $(2\pi)$-periodic function
  $$
  [\mathrm{R}(\vartheta)](\lambda)=\left|\frac{\Theta(\rme^{-\rmi\lambda})\Phi(1)}{\Phi(\rme^{-\rmi\lambda})\Theta(1)}\right|^2-1\;.
  $$
Define the mapping
  \begin{align*}
    \Psi:[0,1/2)\times\Theta_{p,q}&\to[0,1/2)\times\cont\\
    (d,\vartheta)&\mapsto\left(d,f_d\mathrm{R}(\vartheta)\right)    
  \end{align*}
  and denote by $K=\Psi(\tilde K)$ its range over $\tilde K$.
  The following facts are established at the end of this proof section.
  \begin{enumerate}[label=(\roman*),series=fact-proof-corollary-arfima]
  \item\label{item:fact1}  We have $\widetilde{\Lambda}_n=\Lambda_n\circ\Psi$ over $[0,1/2)\times\Theta_{p,q}$.
  \item\label{item:fact2} We have, for all $d\in[0,1/2)$,
    $\vartheta=(\phi,\theta)\in\Theta_{p,q}$ and $\sigma^2>0$,
    $$
    \exp\circ\widetilde\Lambda_n(d,\vartheta) =
    \left|\frac{\Phi(1)}{\Theta(1)}\right|^2\intper{\frac{I_n}{f_d+f_d\mathrm{R}(\vartheta)}} \;.
    $$
  \item\label{item:bicontinuity-R} The mapping $\mathrm{R}$ is continuous and
    one-to-one on
    $\Theta_{p,q}$. We denote by
    $\mathrm{R}^{-1}$ its inverse, which is continuous on  $\mathrm{R}(K_0)$
    for all compact subset $K_0\subset\Theta_{p,q}$.
  \end{enumerate} Then the condition~(\ref{eq:def-param-edstim-arfima})
  defining $(\hat d_n,\hat\vartheta_n)$ is equivalent to
  have~(\ref{eq:def-param-edstim}) with $K_n=K$ and
  $\hat h_n=f_{\hat d_n}\mathrm{R}(\hat \vartheta_n)$ (and the same $\hat d_n$). To
  apply Theorem~\ref{theo:cons} on this sequence $(\hat d_n,\hat h_n)$, we need
  to check that~\ref{item:hyp-consistency} holds with $K_n=K$ and
  $h^*_n=h^*=f_{d^*}R_{\vartheta^*}$ (as in~(\ref{eq:hdtheta-arfima})) for all
  $n$. In this case, only the compactness of $K$ is non-trivial and since
  $\tilde K$ is compact, this compactness follows from the following assertion:
  \begin{enumerate}[label=(\alph*)]
  \item\label{item:remain-to-be-proved} The mapping $\Psi$ is continuous. 
  \end{enumerate}
  Assuming this fact proven, we can apply Theorem~\ref{theo:cons} and get that
  $(\hat d_n,\hat h_n)$ converges a.s. to $(d^*,h^*)=\Psi(d^*,\vartheta^*)$
  (that is, $h^*$ as in~(\ref{eq:hdtheta-arfima})).  Also by
  Fact~\ref{item:fact2} above,
  with~(\ref{eq:c-parm-estim}),~(\ref{eq:csigma2-arfima})
  and~(\ref{eq:sigma-parm-estim-arfima}), if we can apply
  Theorem~\ref{theo:cons}, then we also get that $\hat\sigma^2_n$ is a
  consistent estimator of $\sigma_*^2$.  Finally, it only remains to explain
  how to get that $\hat\vartheta_n$ converges to $\vartheta^*$ a.s. This follows
  from the assertion that $(\hat d_n,\hat h_n)=\Psi(\hat d_n,\hat\vartheta_n)$
  converges a.s. to $(d^*,h^*)=\Psi(d^*,\vartheta^*)$, provided that $\Psi$ can
  be continuously inversed on $K$. To summarize, to conclude the proof, we only need
  to prove the following assertion.
  \begin{enumerate}[label=(\alph*),resume]
  \item\label{item:remain-to-be-proved-bis} The mapping $\Psi$ is bijective and
    bi-continuous from $\tilde K$ to $K$ (its range). 
  \end{enumerate}
  Define the mapping
    \begin{align*}
    \mathrm{A}:(-1/2,1/2)\times\mathcal{A}&\to\cont\\
    (d,h)&\mapsto f_dh\;,    
  \end{align*}
  where, for any $C>0$,
  $$
  \mathcal{A}(C)=\left\{h\in\cont~:~\sup_{t\in\rset^*}|h(t)/t|\leq
    C\right\}\text{ and } \mathcal{A}=\left(\bigcup_{C>0}\mathcal{A}(C)\right) \;.
  $$
  Note that we have, for all $d\in[0,1/2)$, $h\in\mathcal{A}$ and
  $\vartheta\in\Theta_{p,q}$,
  \begin{align*}
    \Psi(d,\vartheta)&=\left(d,\mathrm{A}(d,\mathrm{R}(\vartheta))\right)\;,\\
    \Psi\left(d,\mathrm{R}^{-1}\left(\mathrm{A}(-d,h)\right)\right))&=(d,h)\;.
  \end{align*}
Hence Assertion~\ref{item:remain-to-be-proved-bis} follows from
Assertion~\ref{item:bicontinuity-R} among with the
following facts.
\begin{enumerate}[resume*=fact-proof-corollary-arfima]
\item\label{item:remain-to-be-proved-easy1} For all compact subset $K_0\subset\Theta_{p,q}$ there exists $C>0$ such
  that the range $\mathrm{R}(K_0)\subset\mathcal{A}(C)$. 
\item\label{item:remain-to-be-proved-easy2} For any $C>0$, $\mathrm{A}$ is continuous on
  $(-1/2,1/2)\times\mathcal{A}(C)$. 
\end{enumerate}
\noindent\textbf{Proof of Assertion~\ref{item:fact1}:} This follows directly
from the definitions of $\Lambda_n$ and $\tilde\Lambda_n$
in~(\ref{eq:whittle-contrast}) and~(\ref{eq:whittle-contrast-arfima}) and the
well known fact that, for all $d\in(-1/2,1/2)$ and $\vartheta\in\Theta_{p,q}$,
$$
\intper{\ln(f_d+f_d\mathrm{R}(\vartheta))}=\ln\left|\frac{\Phi(1)}{\Theta(1)}\right|^2\;.
$$

\noindent\textbf{Proof of Assertion~\ref{item:fact2}:} This is simple algebra
using the above definitions. 

\noindent\textbf{Proof of Assertions~\ref{item:bicontinuity-R} and~\ref{item:remain-to-be-proved-easy1}:} Using
standard properties of canonical ARMA processes, we have, for all
$\vartheta=(\phi,\theta)\in\Theta_{p,q}$,
  $$
  [\mathrm{R}(\vartheta)](\lambda)=\left|\frac{\Phi(1)}{\Theta(1)}\right|^2\;\sum_{k\geq1}\alpha_k(\vartheta)\,\left(\cos(k\lambda)-1\right)\;,
  $$
where for any $k\geq1$, the mapping $\vartheta\mapsto\alpha_k(\vartheta)$ is
  polynomial and for any compact subset $K_0\subset\Theta_{p,q}$, there exists
  $C_0>0$ and $\rho_0\in(0,1)$ such that, for all $k\geq1$,
  $$
\left|  \alpha_k(\vartheta)\right|\leq C_0\,\rho^k\;.
$$
Assertion~\ref{item:remain-to-be-proved-easy1} easily follows as well as the
continuity of $\mathrm{R}$ over $\Theta_{p,q}$. Also since
$\mathrm{R}(\vartheta)+1$ is the spectral density of the ARMA($p,q$) process
with ARMA polynomials $\Phi$ and $\Theta$, it is obvious that $\mathrm{R}$ is
one-to-one on $\Theta_{p,q}$. Let $K_0$ be a compact subset of
$\Theta_{p,q}$. Then for all $h\in\mathrm{R}(K_0)$, using standard arguments,
we can express the reciprocal $\mathrm{R}^{-1}(h)$ by
 $$
 \mathrm{R}^{-1}(h)=\arg\min_{\vartheta\in K_0}\mathcal{K}(h;\vartheta)\;,
 $$
 where, for all $\vartheta=(\phi,\theta)\in\Theta_{p,q}$ and $\tilde h\in\cont$,
 $$
 \mathcal{K}(\tilde h;\vartheta)=\intperarg
 {(\tilde h(\lambda)+1)\;\left|\frac{\Phi(\rme^{-\rmi\lambda})}{\Theta(\rme^{-\rmi\lambda})}\right|^2}{\lambda} \;.
$$
Since $(\tilde h,\vartheta)\mapsto\mathcal{K}(\tilde h;\vartheta)$ is
continuous on $\cont\times\Theta_{p,q}$, we get that $\mathrm{R}^{-1}$ is
obviously continuous on $R^{-1}(K_0)$ and Assertion~\ref{item:bicontinuity-R}
is proved.

\noindent\textbf{Proof of Assertion~\ref{item:remain-to-be-proved-easy2}:}
Let $d\in(-1/2,1/2)$ and $h\in\mathcal{A}(C)$ for some positive constant
$C$. Let $\epsilon>0$ such that $[d-\epsilon,d+\epsilon]\subset(-1/2,1/2)$.
Let
$\eta\in(0,\pi/2)$. Then, for all $d'\in[d-\epsilon,d+\epsilon]$ and
$\tilde h\in\mathcal{A}(C)$, we have, using~(\ref{eq:functional-bounds}),
$$
\sup_{|\lambda|\leq\eta}\left|f_d(\lambda)\,h(\lambda)-f_{d'}(\lambda)\,\tilde h(\lambda)\right|\leq 2\,C\,\sup_{|\lambda|\leq\eta}\left( \overline{f}_{(d,\epsilon)}(\lambda)\,|\lambda|\right)\;,
$$
which tends to 0 as $\eta\to0$. On the other hand, we clearly have that
$$
\lim_{(d',\tilde h)\to(d,h)}\sup_{\eta\leq|\lambda|\leq\pi}\left|f_d(\lambda)\,h(\lambda)-f_{d'}(\lambda)\,\tilde h(\lambda)\right|=0\;.
$$
The last two displays yield~\ref{item:remain-to-be-proved-easy2}, which
concludes the proof.

\section{Numerical experiments}
\label{sec:numer-exper}

\subsection{Simulated trawl processes}

We take a sample of exponents $\alpha^*\in(1,2): \alpha^*\in\{1.1,1.3,1.5,1.7,1.9\}$, and for each of them, we
generate two trawl processes obtained from two different sequences $(a_k)$ and two different seed processes: 
\begin{enumerate}
\item Poisson seed:  $\gamma(t)$ is a homogeneous Poisson counting process with
  unit intensity and power sequence~:  $a_k=10.\,(k+1)^{-\alpha^*}$.
\item Binomial seed with $n=10$ (see Section~\ref{sec:other-seeds}) and power sequence~:  $a_k=(k+1)^{-\alpha^*}$.
\end{enumerate}

We showed in Section~\ref{sec:other-seeds} that the
spectral densities of the two resulting trawl processes  are of the
form~(\ref{eq:sp-dens-general}) with $(d^*,h^*)\in[0,1/2]\times H(\alpha^*-1,C)$, for the Poisson (thus Lévy) seed and the  Binomial  seed. 

\subsection{Estimation of the trawl exponent}
To test our new estimator, we will compare it with local Whittle estimator for long range dependent sequences (\cite{Rob}). First, let us recall the definition of the local Whittle estimator.

Here the Hurst exponent is $H=(3-\alpha)/2$  and  the spectral density writes 
$$
f(\lambda)=\frac{c_0}{\Gamma(\alpha)\cos\Big(\frac{\pi (3-\alpha)}2\Big)}\lambda^{\alpha-2}(1+o(\lambda)), 
$$
Let $\lambda_j=2j\pi/n$ denote the canonical  frequencies  for $1\le j\le n/2$,
where $n$ is the sample size. The local
Whittle contrast is defined for a given bandwidth parameter $m\le n/2$ by
\begin{eqnarray*}\label{Rfunction}
R(\alpha)= \ln  \widehat G(\alpha)+\frac{\alpha-2}m \sum_{j=1}^m\ln\lambda_j, \  \widehat G(\alpha)=\frac1m \sum_{j=1}^m\lambda_j^{\alpha-2}
I_n(\lambda_j)\;,
\end{eqnarray*}
where $I_n$ is the usual periodogram, see~(\ref{eq:periodo-def}). Then the
local Whittle estimator $\hat\alpha_{\text{\tiny LW}}$ is computed through numerical minimization of
$R(\alpha)$ over $\alpha\in[1,2]$. In the non-linear case, such as trawl
processes with Poisson or binomial seed, the use of such an estimator is
theoretically justified in\cite{DGH} under the assumption
$\lim_{n\to\infty}\left(\frac mn+\frac1m\right)=0$.

The parametric Whittle estimator that we use is based on the
parameterization~(\ref{eq:sp-dens-general}). Thus we set
$\hat\alpha_{\text{\tiny PW}}=2(1-\hat d_n)$ with $\hat d_n$ the estimator
obtained through numerical minimization of $\Lambda_n(d,h)$ defined
by~(\ref{eq:whittle-contrast}) over $d\in[0,1/2]$ and $h\in\mathcal{P}_N$ (see
Section~\ref{sec:param-whittle-estim}), for a given $N$.

In our setting, both the local Whittle estimator $\hat\alpha_{\text{\tiny LW}}$
and the parametric Whittle estimator $\hat\alpha_{\text{\tiny PW}}$ rely on
tuning parameters, respectively denoted by $m$ and $N$.  Observe that $N$ and
$m$ have very different interpretations. As the bandwidth parameter $m$
increases, a larger range of frequencies is used in the estimation, thus
reducing the variance, and the estimator relies on the approximation
$f(\lambda)\approx c\lambda^{\alpha-2}$ also over a larger range of
frequencies, thus worsening the bias. In contrast, as $N$ increases, we expect
the variance to increase, since the number of parameters to estimate for $h$ is
larger, and the bias to decrease, since the approximation of $h$ by a
trigonometric polynomial is more accurate.

\subsection{Results}
We show here the comparison of the two estimators. We have to guess the
hyperparameter of the two estimators: The ``m'' for the local Whittle and the
number ``$N$'' of Fej{\'e}r kernels for the parametric estimator. We give our
results in function of the choice of these hyperparameters. For each
experiment, we write in bold the choice of hyperparameters minimizing the sum
of the square of the bias and the variance (the mean square error). In all
cases, but especially for the Binomial seed, we can see in the following tables
that our estimator outperforms the local Whittle estimator. A right choice for
the number of kernels seems to be around between $3$ and $5$, even if, best
results may be obtained for higher number of kernels, but this may be due to
local minima reached by numerical optimization.

\begin{table}[!h]\caption{Estimation results for a local Whittle estimator, when $\alpha\in \{1.1, 1.3, 1.5, 1.7, 1.9\}$,\\ $5000$ observations, $m\in \{20,50,100,200\}$, $100$ replications.} \vspace{-0.5cm} \footnotesize{ \begin{center} \begin{tabular}{|c|c|c|c|c|c||c|c|c|c|}\hline\hline $\alpha$   & Statistic  & \multicolumn{ 4}{|c||}{Poisson seed} & \multicolumn{ 4}{|c|}{Binomial seed}  \\  \hline           &  m  & 20 & 50 & 100 & 200 & 20 & 50 & 100 & 200 \\ \hline\hline
1.1            & bias($\widehat{\alpha}_{\text{\tiny LW}})$ & 0.0679  &  0.0179   &  -0.0193  & {\bf -0.0563} &  0.0441  & -0.0332 & {\bf -0.0776} & -0.0953  \\                & sd($\widehat{\alpha}_{\text{\tiny LW}})$ &  0.1973 &  0.1296   &  0.0854  &  {\bf 0.0513}  &   0.1943  & 0.1008 & {\bf 0.0467} & 0.0174 \\
\hline 1.3            & bias($\widehat{\alpha}_{\text{\tiny LW}})$ & 0.0305 & -0.0072 & {\bf -0.0495} & -0.0941 & -0.0887 & {\bf -0.1294} & -0.1795 & -0.2352 \\                  & sd($\widehat{\alpha}_{\text{\tiny LW}})$ &  0.2697  & 0.168 & {\bf 0.1078} & 0.0765  & 0.2383 & {\bf 0.144} & 0.0969 & 0.0635 \\
\hline 1.5            & bias($\widehat{\alpha}_{\text{\tiny LW}})$ & -0.0374 & -0.0726 & {\bf -0.1053} & -0.1402 & -0.0513 & {\bf -0.1595} & -0.224 & -0.2932 \\              & sd($\widehat{\alpha}_{\text{\tiny LW}})$ &  0.2939  & 0.1837 & {\bf 0.111} & 0.075  & 0.2861  & {\bf 0.166} & 0.1051 & 0.0872 \\
\hline 1.7            & bias($\widehat{\alpha}_{\text{\tiny LW}})$ & -0.1025 & -0.1447 & {\bf -0.17} & -0.2074 & -0.1118 & {\bf -0.1998} & -0.2529 & -0.336 \\              & sd($\widehat{\alpha}_{\text{\tiny LW}})$ &  0.2594   & 0.1656 & {\bf 0.1125} & 0.0822  & 0.2509  & {\bf 0.1841} & 0.1159 & 0.0786 \\
\hline 1.9           & bias($\widehat{\alpha}_{\text{\tiny LW}})$ &  -0.1658  & -0.1954 & {\bf -0.2148} & -0.2644 & -0.1955 & {\bf -0.2588} & -0.3238 & -0.4069\\               & sd($\widehat{\alpha}_{\text{\tiny LW}})$ &   0.2238   & 0.1465  & {\bf 0.1125}  & 0.0804  & 0.2614  & {\bf 0.1802} & 0.1265 & 0.0859 \\
\hline \hline \end{tabular}\end{center}}\end{table}

\begin{table}[!h]\caption{Estimation results for the parametric Whittle estimator and Poisson seed, when\\ $\alpha\in \{1.1, 1.3, 1.5, 1.7, 1.9\}$, $5000$ observations, $N\in \{2,3,4,5,6,7,8,9\}$, $100$ replications.} \vspace{-0.5cm} \footnotesize{ \begin{center} \begin{tabular}{|c|c|c|c|c|c|c|c|c|c|}\hline\hline $\alpha$   & Statistic  & \multicolumn{ 8}{|c|}{Poisson seed}  \\  \hline           &  $N$  & 2 & 3 & 4 & 5 & 6 & 7 & 8 & 9 \\ \hline\hline
1.1            & bias($\widehat{\alpha}_{\text{\tiny PW}})$ & {\bf -0.0037}  &  0.0524   &  0.1075  & 0.157 &  0.206  & 0.2595 & 0.2993 & 0.2927  \\                & sd($\widehat{\alpha}_{\text{\tiny PW}})$ &  {\bf 0.058} &  0.0754   &  0.0925  &  0.1114  &   0.1309  & 0.1376 & 0.1585 & 0.1552 \\
\hline 1.3            & bias($\widehat{\alpha}_{\text{\tiny PW}})$ & {\bf -0.0072} & 0.0803 & 0.1483 & 0.2282 & 0.3085 & 0.3639 & 0.362 & 0.3233 \\                  & sd($\widehat{\alpha}_{\text{\tiny PW}})$ &  {\bf 0.0694}  & 0.0922 & 0.1117 & 0.1361  & 0.1569 & 0.1574 & 0.1308 & 0.1288 \\
\hline 1.5            & bias($\widehat{\alpha}_{\text{\tiny PW}})$ & {\bf 0.0159} & 0.1045 & 0.1974 & 0.2766 & 0.3353 & 0.356 & 0.2743 & 0.1949 \\              & sd($\widehat{\alpha}_{\text{\tiny PW}})$ &  {\bf 0.0817}  & 0.1162 & 0.1381 & 0.1439  & 0.1312  & 0.1073 & 0.0918 & 0.0788 \\
\hline 1.7            & bias($\widehat{\alpha}_{\text{\tiny PW}})$ & {\bf 0.0199} & 0.1249 & 0.1959 & 0.2305 & 0.2544 & 0.2383 & 0.125 & 0.0781 \\              & sd($\widehat{\alpha}_{\text{\tiny PW}})$ &  {\bf 0.102}   & 0.1144 & 0.0933 & 0.0774  & 0.0558  & 0.0441 & 0.0788 & 0.0836 \\
\hline 1.9           & bias($\widehat{\alpha}_{\text{\tiny PW}})$ &  0.0199  & 0.0653 & 0.0796 & 0.0841 & 0.0814 & {\bf 0.0469} & -0.0201 & -0.0696\\             & sd($\widehat{\alpha}_{\text{\tiny PW}})$ &   0.0696   & 0.0434  & 0.0216  & 0.0158  & 0.0151  & {\bf 0.0288} & 0.0578 & 0.0602 \\
\hline \hline \end{tabular}\end{center}}\end{table}

\begin{table}[!h]\caption{Estimation results for the parametric Whittle estimator and Binomial seed, when\\ $\alpha\in \{1.1, 1.3, 1.5, 1.7, 1.9\}$, $5000$ observations, $N\in \{2,3,4,5,6,7,8,9\}$, $100$ replications.} \vspace{-0.5cm} \footnotesize{ \begin{center} \begin{tabular}{|c|c|c|c|c|c|c|c|c|c|}\hline\hline $\alpha$   & Statistic  & \multicolumn{ 8}{|c|}{Binomial seed}  \\  \hline           &  $N$  & 2 & 3 & 4 & 5 & 6 & 7 & 8 & 9 \\ \hline\hline
1.1            & bias($\widehat{\alpha}_{\text{\tiny PW}})$ & -0.0892  &  -0.0762   &  {\bf -0.0458}  & -0.0048 &  0.0443  & 0.0829 & 0.1193 & 0.1115  \\                & sd($\widehat{\alpha}_{\text{\tiny PW}})$ &  0.0046 &  0.0316   &  {\bf 0.0602}  &  0.0804  &   0.1084  & 0.1232 & 0.141 & 0.1269 \\
\hline 1.3            & bias($\widehat{\alpha}_{\text{\tiny PW}})$ & -0.2235 & -0.1334 & -0.0651 & {\bf -0.0025} & 0.0579 & 0.1127 & 0.1251 & 0.0826 \\                  & sd($\widehat{\alpha}_{\text{\tiny PW}})$ &  0.0517  & 0.0817 & 0.1032 & {\bf 0.1213}  & 0.1365 & 0.1498 & 0.1499 & 0.1322 \\
\hline 1.5            & bias($\widehat{\alpha}_{\text{\tiny PW}})$ & -0.2472 & -0.1348 & {\bf -0.0467} & 0.042 & 0.1227 & 0.1925 & 0.1365 & 0.0578 \\              & sd($\widehat{\alpha}_{\text{\tiny PW}})$ &  0.0634  & 0.0908 & {\bf 0.1102} & 0.1342  & 0.1533  & 0.1527 & 0.1296 & 0.132 \\
\hline 1.7            & bias($\widehat{\alpha}_{\text{\tiny PW}})$ & -0.2484 & -0.1068 & 0.0152 & 0.1083 & 0.1703 & 0.1808 & {\bf 0.0263} & -0.0468 \\              & sd($\widehat{\alpha}_{\text{\tiny PW}})$ &  0.0781   & 0.0941 & 0.118 & 0.1209  & 0.0998  & 0.0759 & {\bf 0.1009} & 0.0956 \\
\hline 1.9           & bias($\widehat{\alpha}_{\text{\tiny PW}})$ &  -0.268  & -0.1074 & -0.0025 & 0.0439 & 0.0631 & {\bf 0.0278} & -0.1148 & -0.2035\\             & sd($\widehat{\alpha}_{\text{\tiny PW}})$ &   0.0888   & 0.1118  & 0.0922  & 0.0693  & 0.0483  & {\bf 0.0498} & 0.0855 & 0.0903 \\
\hline \hline \end{tabular}\end{center}}\end{table}

\section{Conclusion}
\label{sec:conclusion}
In this paper the consistency of pointwise and broadband spectral
estimators have been proved under general conditions. We show in
particular that a wide class of trawl processes satisfy these
conditions. However, in view of the sample mean behaviors exhibited in
\cite{DJLS-spa}, finer results on the asymptotic behavior of these
estimators should be treated under more specific assumptions. Up to
our best knowledge, very few results are available for non-linear
long-range dependent trawl processes. The rate of a wavelet based
semi-parametric estimator of the long-range dependence parameter is
studied in \cite{fay07} for so called Infinite source Poisson, which
can be seen as a specific trawl process with Poisson seed.  A first
step for future work could be to study the asymptotic behavior of such
an estimator.

\paragraph{Acknowledgements.}
This work has been developed within the MME-DII center of excellence
(ANR-11-LABEX-0023-01) and with the help of PAI-CONICYT MEC
80170072. We would like to thank Donatas Surgailis for fruitful
discussions about this work.



\end{document}